\documentclass[11pt]{amsart}
\usepackage{amsmath, amsfonts, amssymb, amsthm, mathtools, mathrsfs, bm}
\usepackage{etoolbox}
\usepackage{enumerate}
\usepackage[shortlabels]{enumitem}
\usepackage[numbers]{natbib}
\usepackage{setspace}
\usepackage{xcolor}
\usepackage{tikz-cd}
\usepackage{verbatim}

\usepackage[pagebackref,colorlinks,linkcolor=blue,citecolor=blue,urlcolor=blue,hypertexnames=true]{hyperref}

\newtheorem{theorem}{Theorem}[section]
\newtheorem{lemma}[theorem]{Lemma}
\newtheorem{proposition}[theorem]{Proposition}
\newtheorem{corollary}[theorem]{Corollary}

\theoremstyle{definition}
\newtheorem{definition}[theorem]{Definition}
\newtheorem{example}[theorem]{Example}
\newtheorem{remark}[theorem]{Remark}
\newtheorem{notation}[theorem]{Notation}

\newcommand{\MD}[1]{{\color{blue}#1}}

\newcommand{\G}{\Gamma}
\newcommand{\La}{\Lambda}

\newcommand{\ep}{\varepsilon}

\newcommand{\N}{\mathbb{N}}
\newcommand{\Z}{\mathbb{Z}}
\newcommand{\Q}{\mathbb{Q}}
\newcommand{\R}{\mathbb{R}}
\newcommand{\C}{\mathbb{C}}
\newcommand{\HH}{\mathbb{H}}

\DeclareMathOperator{\im}{Im}
\DeclareMathOperator{\ch}{ch}
\DeclareMathOperator{\Ext}{Ext}
\DeclareMathOperator{\coker}{coker}

\newcommand{\K}{\mathcal{K}}

\newcommand{\Tr}{\mathrm{Tr}}
\newcommand{\Aut}{\mathrm{Aut}}

\DeclareMathOperator{\Hom}{Hom}

\newcommand{\id}[1]{\mathrm{id}_{#1}}

\newcommand{\tensor}{\otimes}
\newcommand{\openp}{\pmb{(}}
\newcommand{\closep}{\pmb{)}}
\newcommand{\Proj}[1]{\mathcal{P}r(#1)}
\newcommand{\uProj}[1]{\mathcal{P}r_0(#1)}
\renewcommand{\geq}{\geqslant}
\renewcommand{\leq}{\leqslant}

\definecolor{darkpastelred}{rgb}{0.76, 0.23, 0.13}
\definecolor{darkred}{rgb}{0.55, 0.0, 0.0}
\definecolor{darkmagenta}{rgb}{0.55, 0.0, 0.55}
\definecolor{coolblack}{rgb}{0.0, 0.18, 0.39}
\definecolor{ceruleanblue}{rgb}{0.16, 0.32, 0.75}
\let\le\leqslant
\let\leq\leqslant

\title[Almost representations]{On Chern classes of almost representations
}
\author{Marius Dadarlat and Forrest Glebe}
\date{\today}
\thanks{M.D. was partially supported by NSF grant \#DMS-2247334
\\
F.G. was supported by a Ross-Lynn grant in the 2023-2024 academic year.}
\begin{document}
	
	\maketitle
   
\begin{abstract}
   For a discrete group $\Gamma$, we study vector bundles $E_\rho$ on compact subsets of $B\Gamma$ associated to almost representations $\rho:\Gamma \to U(n)$. We compute the first Chern class of $E_\rho$ in terms of $\rho$. When $\rho$ is both projective and almost multiplicative, we determine its Chern character. These invariants yield obstructions to perturbing almost representations to those arising from projective representations. For residually finite amenable groups, the $K$-theory classes of $E_\rho$ classify almost representations up to stable equivalence. Finally, for $\mathbb{Z}^d$, $\mathbb{Z}\times \mathbb{H}_3$, and $\mathbb{H}_3\times \mathbb{H}_3$, we construct explicit almost representations with prescribed Chern classes.
\end{abstract}
\tableofcontents
		\section{Introduction}

	Let $\G$ be a discrete countable group. 
Let $S$ be a finite subset of $\G$ and let $\delta>0$.
    A unital map $\rho:\G \to U(n)$ is called an $(S,\delta)$-representation if
  \[\|\rho(ab)-\rho(a)\rho(b)\|<\delta, \,\, \forall a,b \in S.\] 
  We call such a map an almost representation, even when $S$ and $\varepsilon$ are not explicitly specified, provided that $S$ is sufficiently large and $\varepsilon$ sufficiently small for the context. In the literature, the same notion is also referred to as a quasi-representation, \cite{CGM:flat}, \cite{BB}.
   A sequence of unital maps $\{\rho_n:\G \to U(k_n)\}_n$ will be called an asymptotic representation of $\G$ if $$\lim_{n\to \infty} \|\rho_n(ab)-\rho_n(a)\rho_n(b)\|=0, \,\,\forall a,b \in \G.$$
   The study of almost representations of groups with respect to the operator norm goes back to Kazhdan \cite{Kazhdan-epsilon}, Voiculescu \cite{Voi:unitaries}, Connes, Gromov and Moscovici \cite{CGM:flat} and Exel and Loring \cite{Exel-Loring:inv}.
  We study asymptotic and almost representations using K-theory methods and quasidiagonality. Vector bundles on compact subspaces of the classifying space $B\G$ play a key role. The idea of constructing almost representations from almost flat bundles was introduced in \cite{CGM:flat}.
   The approximate monodromy correspondence between almost flat bundles and almost representations was further studied in \cite{Hanke-Schick}, \cite{Hanke},
\cite{Carrion-Dadarlat}, \cite{Hunger} and \cite{Kubota3}.
 A functional-analytic approach to constructing almost representations and almost flat bundles was introduced in \cite{BB} and \cite{AA}, using the notion of quasidiagonal $K$-homology classes for group $C^*$-algebras. This method draws upon key results by Kasparov \cite{Kas:inv}, Yu \cite{Yu:BC}, and Tu \cite{Tu:gamma}, particularly in relation to the Novikov and Baum-Connes conjectures. The applicability of this technique, revisited in \cite{DadCarrion:almost_flat}, was significantly broadened by Kubota in \cite{Kubota2} through the consideration of quasidiagonal $C^*$-algebras that are intermediate between the full and the reduced group $C^*$-algebras; see also \cite{CCC}, \cite{EEE} for further work in this direction with applications to non-stability.
  
It is convenient to work with a CW-complex model of $B\G$ as described in \cite[Example 1B.7]{hatcher} as we elaborate in Section~\ref{sec:2}.  This allows us to identify the chain complex which defines the simplicial homology of  $B\G,$  with the non-homogeneous bar complex that defines the homology of $\G$. We will also fix an exhaustion of $B\G$ by an increasing sequence of finite subcomplexes $(Y_n)_n.$
Let us recall that Atiyah-Segal's construction which associates to a finite dimensional representation of $\G,$ a flat vector bundle on $B\G$  generalizes to almost representations. If the classifying space $B\G$ admits a realization as a finite CW-complex,
   then, there exist a finite subset $S$ of $\G$  and  $\delta>0$ such that one can associate to any $(S,\delta)$-representation $\rho:\G \to U(n)$, a rank-$n$  almost flat bundle $E_\rho$ over $B\G$, see for example \cite{AA}, \cite{EEE}. The construction of $E_\rho$ will be reviewed subsequently. For groups $\G$ that are quasidiagonal and admit a $\gamma$-element, it is known that $K^0(B\G)$ is generated by almost flat vector bundles $E_\rho$ as above, \cite{AA}, \cite{Kubota1}, \cite{CCC}. Since flat complex vector bundles have trivial rational Chern classes, \cite{Milnor}, it follows that if $[E_\rho]-[n]\neq 0$ in $K^0(B\G)\otimes \Q$, then
  the almost representation $\rho$  is not a small perturbation of  a genuine representation, see \cite{CCC}. Specifically, there are $\ep,\delta_0>0$ and  finite sets $F,S_0\subset \G$ such that  if $\rho:\G \to U(n)$ is an $(S,\delta)$-representation  with $\delta\leq\delta_0$ and  $S \supseteq S_0$, then for any homomorphism $\pi:\G \to U(n),$
  \[\max_{s\in F} \|\rho(s)-\pi(s)\|\geq \ep.\]
  
  One can extend the construction above to general groups with $B\G$ an infinite CW complex. Thus, for any compact subspace $Y\subset B\G,$ there are $S$ and $\delta$ such that to $(S,\delta)$-representation $\rho:\G \to U(n)$ we can associate a rank-$n$ vector bundle $E_\rho^Y$ over $Y$, see Section~\ref{Proof1}.
 
  This leads naturally to the question of computing invariants of the bundle $E_\rho^Y$ in terms of $\rho$. Using previous work \cite{BB}, \cite{DDD}, \cite{arXiv:2204.10354}, \cite{EEE}, we compute the first Chern class of $E^Y_\rho$ as described below. Moreover, if $\rho$ is not only an almost representation, but also a projective representation, in the sense that {$\rho(a)\rho(b)\rho(ab)^{-1}=\lambda(a,b)1_n,$ with $\lambda(a,b)\in \mathbb{T}$ for all $a,b\in \G$, }
   then we compute the rational Chern classes of $E^Y_\rho$.
  Let  $j$ denote the canonical isomorphism  $j:H^2(\G,Q)\to H^2(B\G,Q),$ with $Q$ an abelian group. As explained in Section~\ref{sec:2}, we shall make no notational distinction between and element $x$ and its image $j(x)$. A local 2-cocyle on $\G$ is a map that satisfies  the 2-cocycle equation for finitely many triples $a,b,c \in \G.$
  This notion is discussed in Section~\ref{sec:2} along with the corresponding map $j:Z^2(\G,Q)_Y \to Z^2(Y,Q)$.
  We write $c_i(E)_{\Z}$ for the i$^{th}$ integral Chern class and $c_i(E)$ for the i$^{th}$ Chern class considered as a rational or real cohomology class, depending on context.
  
 The first half of the paper is devoted to computation of invariants of $E^Y_\rho$.
  \begin{theorem}\label{thm:c11} Let $\G$ be a countable discrete group. For any finite subcomplex $Y$ of $B\G$,
		there are  a finite set $S\subset \G$ and $\delta>0$ such that for any $(S,\delta)$-representation $\rho:\G \to U(n),$ the vector bundle $E^Y_\rho$ is well-defined and
		
		(1)  The equation
\begin{equation}\label{coco}
  \omega(a,b):=\frac{1}{2\pi i}\Tr(\log(\rho(a)\rho(b)\rho(ab)^{-1})),
\end{equation}
		defines a local 2-cocycle $\omega \in Z^2(\G,\R)_{Y}$  with the property that
		$c_1(E^Y_\rho)=[\omega]$ in $H^2(Y,\R).$
		{Moreover, for any integral 2-cycle  $c\in Z_2(Y,\Z),$ the corresponding Kronecker pairing takes integral values, $\langle[\omega], [c]\rangle\in \Z$.}
		
		(2)  If $\rho:\G \to U(n)$ is both an $(S,\delta)$-representation and a projective representation, then the total Chern class of $E_\rho^Y$ is
  $$c(E_\rho^Y)=  \left(1+\frac{1}{n}c_1(E^Y_\rho)\right)^n=\left(1+\frac{1}{n}[\omega]\right)^n$$   
        and  the Chern character is $$\ch(E^Y_\rho)=ne^{\frac{1}{n}c_1(E^Y_\rho)}=n \, e^{{ \frac{1}{n}[\omega]}}.$$
			\end{theorem}
	 When $B\G$ admits a compact model, $E_\rho$ can be constructed as a bundle over $B\G$ and $[\omega]\in H^2(B\G,\R)$.  In general,
 to detect the nontriviality of $c_1(E^Y_\rho),$
   we pair $[\omega]$ with 2-homology classes. This recovers the winding number invariants of the type used in \cite{Kazhdan-epsilon}, \cite{Exel-Loring:inv} and \cite{ESS-published}.
  The Exel-Loring formula \cite{Exel-Loring:inv}  shows the equality of two  invariants associated to a pair of almost commuting unitaries regarded as an almost representation of $\Z^2$.
A generalization of the Exel-Loring formula to almost representations of
arbitrary discrete groups $\G$ was given in \cite{DDD} in the form of an index formula.  This formula, which we will review in Theorem~\ref{thm:ddd}, defines 
a pairing $\pmb{\langle} \rho, r\pmb{\rangle}$ between sufficiently multiplicative almost representations \(\rho: \Gamma \to U(k)\) and elements $[r]$ of the group homology  \(H_2(\Gamma, \mathbb{Z})\) realized through the Hopf formula for 2-homology. We write this symbolically as
\[\{\text{Almost representations}\}\times H_2(\G,\Z)^{Hopf}\to \Z,\quad  (\rho,[r])\mapsto \pmb{\langle} \rho, r\pmb{\rangle}.\]

Another pairing $\openp \rho ,c\closep,$ between almost representations \(\rho: \Gamma \to U(k)\) and elements of \(H_2(\Gamma, \mathbb{Z}),\) was introduced in \cite{arXiv:2204.10354} using the bar-resolution definition of \(H_2(\Gamma, \mathbb{Z})\): 
\[\{\text{Almost representations}\}\times H_2(\G,\Z)\to \Z,\quad  (\rho,[c])\mapsto \openp \rho,c\closep.\] 
We review this second pairing in equation   \eqref{eq:111f} below. 
In Section~\ref{sec:pairings} we show that the two pairings can be identified modulo the isomorphism $\varphi:H_2(\Gamma,\mathbb{Z})^{Hopf} \to H_2(\Gamma,\mathbb{Z})$, in the sense that \[\pmb{\langle} \rho, r\pmb{\rangle}=\openp \rho,\varphi(r)\closep.\]
 Furthermore, we give the following geometric interpretation of the pairings.
If $B\G$ admits a compact model. Then 
 there exist $\delta>0$ and a finite set $S\subset \G$ such that for any $(S,\delta)$-representation  $\rho:\G\to U(n),$  $\openp \rho,c\closep$ coincides with the Kronecker pairing between the first Chern class of $E_\rho$ and $[c]\in H_2(B\G,\Z)$:
\[\openp \rho,c\closep=\langle c_1(E_\rho), [c]\rangle.\]
If $B\G$ is not compact, we have a similar interpretation that involves $c_1(E_\rho^Y),$ where $Y$ is a finite subcomplex of $B\G$ that supports the 2-cycle $c$, as discussed in the proof of Theorem~\ref{thm:c1}. 
\vskip 8pt
 {In the second half of the paper we study the extent to which the bundle $E_\rho$ determines $\rho$.} 
  We establish the following stable uniqueness result, showing that the associated vector bundles classify almost representations up to stable approximate unitary equivalence under suitable conditions.
  \begin{theorem}\label{thm:unique}  Let $\G$ be a torsion-free residually finite countable amenable group. For any finite set $F\subset \G$ and any $\ep>0,$ there exist a finite set $S\subset \G$, $\delta>0$ and a compact subspace $Y\subset B\G$
such that 
  for any two $(S,\delta)$-representations
$\rho,\rho':\G \to U(k)$ with $[E_\rho^Y]=[E_{\rho'}^Y]$ in $K^0(Y),$ there is a representation $\pi:\G \to U(m)$ and a unitary $u\in U(k+m)$ such that 
\begin{equation}\label{eq:ssu}
\|u(\rho(s)\oplus \pi(s))u^*-\rho'(s)\oplus \pi(s)\|<\ep,\quad \forall s\in F.
	\end{equation}
    \end{theorem}
    The same conclusion remains valid even if $\Gamma$ is not residually finite; however, in that case we need to allow $\pi$ to be an $(S,\delta)$-representation.
   To complete the picture, we note that the following existence result is implicitly contained in \cite{AA}. The class of the trivial bundle of rank $k$ is denoted by $[k].$
   Fix an exhaustion of $B\G$ by an increasing sequence of finite connected subcomplexes $(Y_n)_n.$
   Kasparov \cite{Kasparov-conspectus}, uses the notation $LK^*(B\G)=\varprojlim K^*(Y_n).$
   \begin{theorem}  Let $\G$ be a torsion free residually finite countable amenable group.
        For any $(z_n)_n\in \varprojlim \widetilde{K}^0(Y_n) $, there exist an asymptotic representation $\{\rho_n:\G \to U(k_n)\}_n$ and finite dimensional representations $\{\pi_n:\G \to U(k_n)\}_n$, such that $[E_{\rho_n}]-[k_n]=z_n$ for all $n\geq 1$.
  Moreover,
       by Theorem~\ref{thm:unique}, $\{\rho_n\}_n$ is unique up to stable approximate unitary equivalence. Thus, if $\{\rho'_n:\G \to U(k_n)\}_n$ is another lifting of $(z_n)_n$, then there exist a sequence of representations $\{\pi_n:\G \to U(\ell_n)\}_n$ and unitaries $u_n\in U(k_n+\ell_n)$ such that 
       \[\lim_{n\to \infty} \|\rho_n(s)\oplus \pi_n(s)-u_n\left( \rho'_n(s)\oplus \pi_n(s)\right)u_n^*\|=0,\quad \forall s\in \G.\]  
    \end{theorem}
    If $\Gamma$ is not residually finite, one can still lift $(z_n)_n$ to a pair of asymptotic homomorphisms $\{\rho_n, \pi_n : \Gamma \to U(k_n)\}_n$ such that $[E_{\rho_n}] - [E_{\pi_n}] = z_n$ for all $n \geq 1$.
  \begin{corollary}\label{thm:ex&unique} Let $\G$ be a residually finite amenable group such that $B\G$ admits a finite CW complex model. 
  
  (1) For any $z\in \widetilde{K}^0(B\G)$ there are $\delta_0>0$ and  a finite set $S_0\subset \G$ such that for any  finite set $S_0\subset S \subset \G$ and any $0<\delta<\delta_0$ there is an $(S,\delta)$-representation $\rho:\G \to U(n)$  such that $[E_{\rho}]-[n]=z.$ 
   In particular, for any $x\in \widetilde{H}^{\text{even}}(B\G,\Q)$ there is an $(S,\delta)$-representation $\rho:\G \to U(n)$ such that $\widetilde{\ch}(E_\rho)=qx$ for some $q\in \Q$.
  
  (2) For any finite set $F\subset \G$ and any $\ep>0,$ there is a finite set $S\subset \G$ and $\delta>0$
such that for any two $(S,\delta)$-representations
$\rho,\rho':\G \to U(n)$ with $[E_\rho]=[E_{\rho'}]$ in $K^0(B\G),$ there is a representation $\pi:\G \to U(m)$ and a unitary $u\in U(n+m)$ satisfying \eqref{eq:ssu}.
	\end{corollary}
   
  Using Theorem~\ref{thm:c11} and Corollary~\ref{thm:ex&unique}, we can now address the question of perturbing multiples of approximate representations to direct sums of projective representations for poly-$\Z$ groups. For a map $\rho:\G\rightarrow U(n)$ and $r\in\N$, we denote $\rho^{\oplus r}$ by $r\rho$.
\begin{theorem}\label{thm:proja}
Let $\G$ be a  poly-$\Z$ group. Consider the following properties:						\begin{enumerate}
				\item For every $\varepsilon>0$ and finite subset $F\subset\G$, there exist a finite subset $S\subset\G$ and $\delta>0$ so that for any $(S,\delta)$-representation $\rho$, there are an integer $r>0$,  representations  $\pi_1$ and $\pi_2$, a unitary $u$, and  a projective representation $\psi$, so that
				$$||u(r \rho(s)\oplus\pi_1(s))u^*-\psi(s)\oplus\pi_2||<\varepsilon,\quad\forall s\in F.$$
				\item For every $\varepsilon>0$ and finite subset $F\subset\G$, there exist a finite subset $S\subset\G$ and a $\delta>0$ so that for any $(S,\delta)$-representation $\rho$, there are an integer $r>0$, a unitary $u$,  a representation $\pi,$ and a finite family of projective representations $(\psi_i)$,  so that
				$$\left\|u(r \rho(s)\oplus\pi(s))u^*-\bigoplus_i\psi_i(s)\right\|<\varepsilon\quad\forall s\in F$$
				\item For every $\varepsilon>0$ and finite subset $F\subset\G$, there exist a finite subset $S\subset\G$ and a $\delta>0$ so that for any $(S,\delta)$-representation $\rho$, there are an integer $r>0$, a unitary $u$, and finite families of projective representations $(\varphi_i)$ and $(\psi_k)$ so that
				$$\left\|u\left(r\rho(s)\oplus\bigoplus_i\varphi_i(s)\right)u^*-\bigoplus_k\psi_k(s)\right\|<\varepsilon\quad\forall s\in F$$
			\end{enumerate}
			Then condition (1) is true if and only if $\widetilde{H}^{\text{even}}(\G;\Q)=H^2(\G;\Q)$. Condition (2) is true if and only if $\widetilde{H}^{\text{even}}(\G;\Q)$ is spanned by nonnegative linear combinations of elements of the form $e^x-1$ with $x\in H^2(\G;\Q)$. Condition (3) is true if and only if $\widetilde{H}^{\text{even}}(\G;\Q)$ is spanned by linear combinations of elements of the form $e^x$ with $x\in H^2(\G;\Q)$.
		\end{theorem}

Let $\Gamma$ be a finitely generated group with a non-torsion 2-cohomology class $[\omega]$ that corresponds to a central extension where the middle group is residually finite. 
Motivated by work of {Kazhdan} \cite{Kazhdan-epsilon}, Voiculescu \cite{Voi:unitaries} and Eilers, Shulman and S{\o}rensen \cite{ESS-published}, the second author provided in \cite{arXiv:2204.10354} an explicit formula, in terms of the 2-cocycle $\omega$, for projective representations of $\G$ that are almost representations  and which  are not perturbable to genuine representations.

In the last part of the paper we explore the question of constructing concrete almost representations which realize higher dimensional cohomological invariants.
More precisely, we consider tensor products and direct sums of  projective representations
  to construct almost representations 
   $\rho$ for which the associated almost flat bundle $E_\rho$ is stably isomorphic to a given bundle $E$ over $B\G$ and which in particular will correspond to higher dimensional cohomology classes in $H^{2k}(\G,\Q)$.
  Due to limitations of our approach, we have satisfactory results only for certain classes of groups where the ring $K^0(B\G)$ is generated by line bundles.
  However, Theorem~\ref{thm:unique} shows that almost representations constructed out of projective representations are as general as any other almost representation,  in many cases when the ring $K^0(B\G)$ is generated by line bundles. We have satisfactory results only for certain classes of groups where the ring $K^0(B\G)$ is generated by line bundles such as $\G=\Z^d$, $\G=\Z\times \HH_3 $
    and  $\G=\HH_3\times \HH_3 .$

 For each given element $\alpha$ of  $K^0(B\Z^d)\cong \bigwedge^{even} \Z^d  $ of virtual rank $0$, we construct a concrete almost representation $\rho:\Z^d \to U(n)$
  such that $[E_\rho]=n+\alpha$. Furthermore, we exhibit explicit almost representations 
 $\rho : \G=\HH_3\times \HH_3 \to U(n)$ such that $c_1(E_\rho)=c_2(E_\rho) =0$ while
$c_3(E_\rho)\neq 0$. In this case, the obstruction to perturbing $\rho$ into a true representation lies solely in an element of $H^{6}(\Gamma, \mathbb{Z})$.

\section{Local $2$-cocycles associated to almost representations}\label{sec:2}

We will only use homology and cohomology with coefficients in commutative rings $Q$ such as $\Z$, $\Q$ and $\R$, viewed as trivial $\Gamma$-modules. The reader is referred to \cite[Chapter II.3]{iBrown:book-cohomology} for more background information.
Let $C_k(\Gamma;Q)$ consist of formal linear combinations of elements of~$\Gamma^k$ with coefficients in $Q$. We write a typical element of $C_2(\Gamma;Q)$ as
$$\sum_{j=1}^m k_j[a_j|b_j]$$
with $a_j,b_j\in\Gamma$ and $k_j\in Q$. There are boundary maps $\partial_2:C_2(\Gamma;Q)\rightarrow C_1(\Gamma;Q)$ defined by
$$\partial_2[a|b]=[a]-[ab]+[b]$$
and $\partial_3:C_3(\Gamma;Q)\rightarrow C_2(\Gamma;Q)$ defined by
$$\partial_3[a|b|c]=[b|c]-[ab|c]+[a|bc]-[a|b].$$
Then $H_2(\Gamma;Q):=\ker(\partial_2)/\im(\partial_3)$. An element of $Z_2(\G,Q):=\ker(\partial_2)$ is referred to as a {2-cycle} and an element in $\im(\partial_3)$ is referred to as a {2-boundary}.

Let us recall now the definition of $2$-cohomology $H^2(\G,Q).$
A 2-cocycle $\sigma:\G^2 \to Q$ 
is a function that satisfies the equation
\begin{equation}\label{eq:cocyle}
    \sigma(a,b)+\sigma(ab,c)=\sigma(a,bc)+\sigma(b,c),\quad \text{for all}\quad  a,b,c \in \G.
\end{equation}

We indicate that $\sigma$ satisfies this condition by writing $\sigma\in Z^2(\G,Q).$ A {2-coboundary} is a 2-cocycle  that can be written in the form
$$\sigma(a,b)=\partial\, \gamma (a,b)=\gamma(a)-\gamma(ab)+\gamma(b)$$ for some function $\gamma:\G \to Q$. $H^2(\Gamma;Q)$ is defined to be the group of 2-cocycles, mod the subgroup of 2-coboundaries. The group operation is pointwise addition.
One can normalize a 2-cocycle $\sigma$ by adding to it a boundary element. Specifically, by replacing $\sigma$ by $\sigma+\partial \,\gamma,$ where $\gamma:\G \to Q$ is defined by $\gamma(a)=-\sigma(e,a)$ for $a\in \G,$ we obtain a 2-cocycle satisfying 
\[\sigma(a,e)=\sigma(e,a)=0.\]

The {Kronecker pairing} between a 2-homology class $c$ and a 2-cohomology class $x$ is a bilinear map $H^2(\Gamma;Q)\times H_2(\Gamma;\Z)\to Q$ defined by the formula
$$\langle x,c\rangle=\left\langle\sigma,\sum_{j=1}^m k_j[a_j|b_j]\right\rangle=\sum_{j=1}^m k_j\sigma(a_j,b_j)$$
where $\sigma$ is a 2-cocycle representing $x$, and $\sum_{j=1}^m k_j[a_j|b_j]$ is a 2-cycle representing $c$. The value does not depend on the choice of representatives.

For a discrete group $\G$, we will use the construction of its classifying space $B\G$  based on the notion of $\Delta$-complex. $\Delta$-complexes are defined in \cite[Ch.2\S 2.1]{hatcher} and the construction of $B\G$ if given in
\cite[Example 1B.7]{hatcher}. An $n$-cell of $B\G$ which is the image of the simplex $[1,a_1,a_1a_2,...,a_1\cdots a_n]$ is labeled by $[a_1|a_2|...|a_n].$ The chain complex which defines the simplicial homology of the $\Delta$-complex of $B\G,$
denoted by $C^{\Delta}_*(B\G;Q)$, coincides with the non-homogeneous bar complex that defines the homology of $\G$ as discussed in \cite[p.36]{iBrown:book-cohomology} via the bijection which maps the element $[a_1|a_2|...|a_n]$ of the basis of $C^{\Delta}_n(B\G;Q)$ to the element $[a_1|a_2|...|a_n]$ of $C_n(\G;Q).$
This allows us to identify $H_n^\Delta(B\G;Q)$ with $H_n(\G;Q)$ and 
$H^n_\Delta(B\G;Q)$ with $H^n(\G;Q).$ 
Furthermore, the simplicial homology (cohomology) is isomorphic to the singular homology (cohomology). At the level of chains this isomorphism is induced by the map which takes an $n$-cell to its characteristic map $\Delta^n \to B\G$. We will make no notational distinction between the simplicial homology (cohomology) of a subcomplex $Y$ of $B\G$ and its singular homology (cohomology). 

 Let $\Gamma = \{a_1, a_2, \dots \}$ be a fixed enumeration of  $\Gamma$, and let $S_n = \{a_1, \dots, a_n \}.$ 
For each $n \ge 1$,  define $Y_n$ as the smallest CW subcomplex of $B\Gamma$ that contains all cells of the form $[b_1| b_2|\dots| b_k]$, where $b_i \in S_n$, $1 \leq k \leq n.$  It is then clear that $Y_n$ is a finite subcomplex of $B\G$ of dimension $n$.
Moreover, $Y_n \subset Y_{n+1}$ and 	$\bigcup_{n=1}^{\infty} Y_n = B\G$. 
 
$B\G$ is endowed with the direct limit topology (the weak topology). In particular any compact subspace of $B\G$ is contained in some $Y_n$~\cite[Proposition A.1]{hatcher}. In this model $B\G$ is not locally finite, hence not locally compact in general. For  a locally finite model one replaces $Y_n$ by $Y_{n}\times[0,1]$ and employs the usual telescope construction where $Y_n\times \{1\}$ is identified with the corresponding subspace of $Y_{n+1}\times \{0\}$.
It is this latter model that we will use in Section~\ref{sec:classifi} in order to appeal to Kasparov's theory $RKK(B\G;A,B).$

 Let $Y$ be a finite $\Delta$-subcomplex of $B\Gamma$ and let $Q$ be an abelian group. We have a natural map of chain complexes $C^{\Delta}_*(Y;Q)\to C^{\Delta}_*(B\G;Q).$

\begin{lemma}\label{2cy-cocy}
Denote by $Z_n(\G; Q)_Y$ the n-cycles in $C^{\Delta}_n(Y;Q)$.

  (a) Every simplicial cycle $c \in Z_2(\G; Q)_Y$ can be directly viewed as a cycle $c\in Z_2(\Gamma; Q)$.
  
  (b) Any map $\omega:\{(a,b) \in \Gamma \times \Gamma \mid [a|b] \in Y_{[2]}\}\to Q,$
 that satisfies the 2-cocycle equation \eqref{eq:cocyle} for all 3-simplices $[a| b| c] \in Y_{[3]},$ defines a 2-cocycle in $Z_\Delta^2(Y;Q)$ and hence an element $[\omega]\in H^2(Y,Q)$. We call $\omega$ a local 2-cocycle and the set of all such maps is denoted by $Z^2(\G;Q)_Y.$
 
 (c) For any $c\in Z_2(\G,\Z),$ there is a finite $\Delta$-subcomplex $Y$ of $B\G$ that supports $c$. If $\omega\in Z^2(\G;Q)_Y$ is as in (b), then we have a pairing $\langle \omega, c\rangle$ given by the Kronecker pairing $H^2(Y,Q)\times H_2(Y,\Z)\to Q$.
\end{lemma}

\begin{proof}
  (a) Since $Y \subseteq B\Gamma$ as $\Delta$-complexes, we have $C^{\Delta}_2(Y;Q) \subseteq C_2^{\Delta}(B\Gamma;Q).$
The boundary operator $\partial_2^Y$ is simply the restriction of $\partial_2^{B\Gamma}$, so cycles remain cycles under this inclusion: $Z_2(Y; Q) \subseteq Z_2(B\Gamma; Q)$.
 The isomorphism $C^{\Delta}_*(B\Gamma;Q) \cong C_*^{\mathrm{bar}}(\Gamma; Q)$ identifies the $\Delta$-complex of $B\Gamma$ with the bar resolution chain complex used in group homology, \cite[p.18]{iBrown:book-cohomology}.
  
  (b) The {cocycle condition} requires that for every 3-simplex $[a| b| c] \in Y_{[3]}$:
$$\omega(b, c) - \omega(ab, c) + \omega(a, bc) - \omega(a, b) = 0.$$
Since $Y$ is a subcomplex, whenever a 3-simplex belongs to $Y$, all its 2-dimensional faces are also in $Y$. This guarantees that $\omega$ is well-defined on all terms appearing in the cocycle equation, and the extension to a 2-cochain $\omega: C_2(Y;Q) \to Q$ via $\omega([a|b]) = \omega(a,b)$ automatically satisfies $\delta_Y^2(\omega) = 0$, so that $\omega \in Z^2(Y,Q)$.

(c) This follows from the previous discussion.
\end{proof}
\begin{lemma}
    If $B\G$ admits a compact model $X,$ then the standard CW model of $B\G$ has a finite subcomplex $Y$ that contains a compact space homotopy equivalent to $X,$ so that if $\iota:Y\rightarrow B\G$ is the inclusion map, then $\iota^*:H^k(B\G;Q)\rightarrow H^k(Y;Q)$ is injective, and $\iota_*:H_k(Y;Q)\rightarrow H_k(B\G;Q)$ is surjective.
\end{lemma}
\begin{proof}
    Let $f:X\rightarrow B\G$ be a homotopy equivalence then let Y be a finite subcomplex containing $f(X)$. Then $f^*$ and $f_*$ are isomorphisms that factor through the cohomology and homology of $Y$.
\end{proof}
The upshot is that if $B\G$ has any compact model, we can understand the cohomology of $\G$, by seeing what it does on the finite subset that is needed to build $Y$ as a CW complex. 

We will prove the following facts for matrices, but using the de la Harpe-Skandalis determinant, they may be generalized to any tracial Banach algebra; see~\cite{partA}. Many results in this section follow from results in~\cite{partA}, but we are providing alternate proofs that lead to explicit numerical bounds.

\begin{definition} Let $S$ be a finite subset of $\G$ and let $\delta>0$.
    A unital map $\rho:\G \to U(n)$ is called an $(S,\ep)$-representation if
  \[\|\rho(st)-\rho(s)\rho(t)\|<\delta, \,\, \forall s,t \in S.\]    \end{definition}

\begin{definition}\label{localCocycle}
If $\rho$ is an $(S,\delta)$-representation with $0<\delta<1$, then for $a,b\in S$ we define the local 2-cocyle associated to $\rho$ by the formula
\begin{equation}\label{def-local-cocycle}
	\omega(a,b)=\omega_\rho(a,b)=
	\frac{1}{2 \pi i}  \mathrm{Tr}\left(\log \left(\rho(a)\rho(b)\rho(ab)^{-1}\right)\right).
\end{equation}
where $\log$ is defined to be the usual power series centered at 1. The terminology is justified by Proposition ~\ref{prop:2cocycle} below.
\end{definition}

\begin{lemma}\cite[Lemma 3.2]{arXiv:2204.10354}\label{loghom}
If $u_1$, $u_2$ are unitary matrices each within $\frac12$ of the identity then $\Tr(\log(u_1u_2))=\Tr(\log(u_1))+\Tr(\log(u_2))$.
\end{lemma}

\begin{proposition}\label{prop:2cocycle}
Suppose that $\rho$ is an $(S,\delta)$-representation. If 
 $a,b,c,ab,bc\in S$ and $\delta\leqslant\frac12$, then $\omega$ satisfies the coycle equation at $a,b,c$:
$$\omega(a,b)+\omega(ab,c)=\omega(a,bc)+\omega(b,c).$$
\end{proposition}

\begin{proof}
Note that
$$\Tr((xy)^n)=\Tr(x(yx)^{n-1}y)=\Tr((yx)^{n-1}yx)=\Tr((yx)^n).$$
Since $\log$ is a power series, it follows that $\Tr(\log(xy))=\Tr(\log(yx))$. Using Lemma~\ref{loghom} and the observation above,
\begin{align*}
\omega(a,b)+\omega(ab,c)&=\frac{1}{2\pi i}(\Tr(\log(\rho(a)\rho(b)\rho(ab)^{-1}))+\Tr(\log(\rho(ab)\rho(c)\rho(abc)^{-1})))\\
&=\frac{1}{2\pi i}\Tr(\log(\rho(a)\rho(b)\rho(c)\rho(abc)^{-1}))\\
&=\frac{1}{2\pi i}\Tr(\log(\rho(b)\rho(c)\rho(abc)^{-1}\rho(a)))\\
&=\frac{1}{2\pi i}(\Tr(\log(\rho(b)\rho(c)\rho(bc)^{-1}))+\Tr(\log(\rho(bc)\rho(abc)^{-1}\rho(a))))\\
&=\frac{1}{2\pi i}(\Tr(\log(\rho(b)\rho(c)\rho(bc)^{-1}))+\Tr(\log(\rho(a)\rho(bc)\rho(abc)^{-1})))\\
&=\omega(a,bc)+\omega(b,c).
\end{align*}
    
\end{proof}

Thus for a given subcomplex $Y\subset B\G,$
$\omega_\rho$ defines a local 2-cocycle in $Z^2(\G,\R)_Y$, 
whenever $S$ is sufficiently large and $\ep$  is sufficiently small. We say that $\omega_\rho$ is the local cocycle associated to $\rho$. 

Denote by $\mathrm{Rep}^n_{(S,\delta)}(\G)$ the set of $(S,\delta)$-representations $\rho:\G \to U(n)$ and
  let $\mathrm{Rep}_{(S,\delta)}(\G)$ be the disjoint union $\bigsqcup_{n\geq 1}\mathrm{Rep}^n_{(S,\delta)}(\G).$
 \begin{proposition}\label{well-defined}
   For any finite subcomplex $Y\subset B\G$ there are a finite subset $S\subset \G$ and $\delta>0$
   such that the correspondence $\rho\mapsto \omega_\rho$ is a well-defined map $\mathrm{Rep}^n_{(S,\delta)}(\G)\to Z^2(\G,\R)_Y.$
 \end{proposition}
 \begin{proof} This follows from Lemma~\ref{2cy-cocy}(b) and Proposition~\ref{prop:2cocycle}.
\end{proof}

    \section{Pairings in tracial C*-Algebras}\label{sec:pairings}

We confine our discussion of pairing to matrix-valued approximate representations.
 The general case of tracial algebras is discussed in \cite{EEE}.
The Exel-Loring \cite{Exel-Loring:inv} formula establishes the equality of two  invariants associated to a pair of almost commuting unitaries regarded as an almost representation of $\Z^2$.
A generalization of the Exel-Loring formula to almost representations of
arbitrary discrete groups $\G$ is given in \cite{DDD}.  The index formula, which we review in Theorem~\ref{thm:ddd} below, defines 
a pairing $\pmb{\langle} \rho, r\pmb{\rangle}$ between sufficiently multiplicative almost representations \(\rho: \Gamma \to U(k)\) and elements $[r]$ of the group homology  \(H_2(\Gamma, \mathbb{Z})\) realized through the Hopf formula, see equation~\eqref{111d} below. 
\[\{\text{Almost representations}\}\times H_2(\G,\Z)^{Hopf}\to \Z,\quad  (\rho,[r])\mapsto \pmb{\langle} \rho, r\pmb{\rangle}.\]

Another pairing $\openp \rho ,c\closep,$ between almost representations \(\rho: \Gamma \to U(k)\) and elements of \(H_2(\Gamma, \mathbb{Z}),\) was introduced in \cite{arXiv:2204.10354} using the bar-resolution definition of \(H_2(\Gamma, \mathbb{Z})\), see equation \eqref{eq:111f} below. 
\[\{\text{Almost representations}\}\times H_2(\G,\Z)\to \Z,\quad  (\rho,[c])\mapsto \openp \rho,c\closep.\] 
We show that the two pairings can be identified via the isomorphism $\varphi:H_2(\Gamma,\mathbb{Z})^{Hopf} \to H_2(\Gamma,\mathbb{Z})$, in the sense that \[\pmb{\langle} \rho, r\pmb{\rangle}=\openp \rho,\varphi(r)\closep.\]

While this fact can be derived from \cite[Prop. 4.1]{EEE}, in the sequel, we will give a proof that allows for quantitative estimates.
The isomorphism of the two pairings is useful because, depending on the context, it can be advantageous to use the form of the pairing that is best suited to the task at hand.

{Furthermore, we give a geometric interpretation of the pairings.
This is easier to explain if $B\G$ admits a compact model. Under that assumption, 
 there exist $\delta>0$ and a finite set $S\subset \G$ such that for any $(S,\delta)$-representation  $\rho:\G\to U(k),$ the pushforward of the Mishchenko line-bundle is a vector bundle  $E_\rho$ on $B\G,$ see \cite{BB}, \cite{AA}.   In this situation we show that  $\openp \rho,c\closep$ coincides with the Kronecker pairing between the first Chern class of $E_\rho$, $c_1(E_\rho)\in H^2(B\G,\Q)$ and $[c]\in H_2(B\G,\Z)$:
\[\openp \rho,c\closep=\langle c_1(E_\rho), [c]\rangle.\]
This equality will be derived as a consequence of Theorem~\ref{thm:c11}
which shows that $c_1(E_\rho)=[\omega_\rho],$ where $\omega_\rho$ is the local 2-cocycle associated to $\rho$.}

Hopf's formula expresses the second homology of \(\Gamma\)  
in terms of a free presentation 
\[1 \to R \to F \stackrel{q}{\longrightarrow} \Gamma \to 1,\] 
where \(q(a)=\bar{a}\), as 
\[H_2(\Gamma,\mathbb{Z})^{Hopf}=\frac{R\cap [F,F]}{[R,F]}.\] 
Each element \(r \in H_2(\Gamma,\mathbb{Z})\) can be represented by a product of commutators \(\prod_{i=1}^{g} [a_i,b_i]\) with \(a_i,b_i \in F\), for some integer \(g \geq 1\), such that \(\prod_{i=1}^{g} [\bar{a}_i,\bar{b}_i] = 1\).
Recall the canonical isomorphisms  $H^2(\G,\Q)\to H^2(B\G,\Q)$ and $H_2(\G,\Z)\to H_2(B\G,\Z).$ The following composition of maps will play a role in the sequel.
\[\begin{tikzcd}
			H_2(\Gamma,\mathbb{Z})^{Hopf}\ar[r, "\varphi"] & H_2(\Gamma,\mathbb{Z})
		\ar[r, "\cong"] &   H_2(B\G,\Z)
	\end{tikzcd}\]
 
 The Chern character in $K$-homology yields an isomorphism
	  $$
	  \operatorname{ch}_* \otimes \id{ \Q }: K_*(B \Gamma) \otimes_{ \Z } \Q \xrightarrow{\cong} H_*(B \Gamma ; \Q )=H_*(\Gamma ; \Q )
	  $$
	Matthey \cite{MR1951251},  \cite{MR2041902}, see also \cite{Bettaieb-Matthey-Valette}, constructed a natural rationally injective homomorphism
	$$
	\beta_2^\G: H_2(B \Gamma ; \Z ) \longrightarrow K_0(B \Gamma),
	$$
 which  is   rationally a  right-inverse of the Chern character:
	\begin{equation}\label{eqn:inverse}
	(\operatorname{ch} \otimes \id{ \Q }) \circ(\beta_2^\G \otimes \id{\Q})=\id{H_2(B\G ; \Q )} 
	\end{equation}
For simplicity we will write $\beta^\G$ in place of $\beta_2^\G$ and denote by $\beta^Y$ its restriction to subspaces $Y$ of $B\G.$

 Let $\alpha^\G : H_2(\G,\Z) \to K_0(\ell^1(\G))$ be the composition \mbox{$\alpha^\G=\mu_1^\G  \circ \beta^\G$} 
 \[\begin{tikzcd}
			H_2(B\G,\Z) \ar[r, "\beta^\G"] & K_0(B\G)
		\ar[r, "\mu_1^\G"] &   K_0(\ell^1(\G))
	\end{tikzcd}\]
 where $\mu_1^\G$ is  the $\ell^1$-version of the assembly map of Lafforgue \cite{Lafforgue}.
We abbreviate the winding number of a closed loop $L$ in $\C\setminus\{0\}$ by $\mathrm{wn} (L)$.
The  linear extension $\rho:\ell^1(\G)\to M_k(\C)$ of a sufficiently multiplicative unital map
 $\rho:\G \to U(k)$ satisfies the following generalization of the Exel-Loring formula: 
 \vskip 4pt
\begin{theorem}[\cite{DDD}]\label{thm:ddd} 
\emph{ Let $r\in H_2(\G,\Z)^{Hopf}$ be represented by  a product of commutators
 $\prod_{i=1}^{g} [a_i,b_i]$ with $a_i,b_i \in F$ and $\prod_{i=1}^{g} [\bar{a}_i,\bar{b}_i]=1$.
 There exist a finite set $S \subset \G$  and $\delta>0$ such that if \(\rho: \Gamma \to U(k)\) is an $(S,\delta)$-representation, then}
 \small{ \begin{equation*}
  \rho_\sharp( \alpha^\G(\varphi(r)))=\mathrm{wn}
  \det\left((1-t)1_n+t\prod_{i=1}^{g} [\rho(\bar{a}_i),\rho(\bar{b}_i)] \right) =\frac{1}{2\pi i}\mathrm{Tr} \log \left(\prod_{i=1}^{g} [\rho(\bar{a}_i),\rho(\bar{b}_i)]\right)
 \end{equation*}
}
 \end{theorem}
 Here, if we write $\alpha^\G(\varphi(r)))=[p_0]-[p_1],$ where $p_i$ are projections in matrices over $\ell^1(\G)$, then  $\rho_\sharp( \alpha^\G(\varphi(r)))) = \rho_\sharp(p_0) -
 \rho_\sharp(p_1),$ where $\rho_\sharp(p_i)\in \Z$ is the rank of the perturbation of $(\mathrm{id}\otimes \rho )(p_i)$ to a projection via analytic functional calculus.
 
 The right-hand side of Theorem~\ref{thm:ddd} 
 defines 
a pairing
\begin{equation}\label{111d}
\bm{\langle} \rho, r \bm{\rangle}
=\frac{1}{2\pi i}\mathrm{Tr} \log \left(\prod_{i=1}^{g} [\rho(\bar{a}_i),\rho(\bar{b}_i)]\right)
\end{equation}
  between almost representations \(\rho: \Gamma \to U(k)\) and elements $r$ of the group homology  \(H_2(\Gamma, \mathbb{Z})^{Hopf}\) realized through the Hopf formula. 
  
  Another pairing $\openp \rho ,c\closep$ between almost representations \(\rho: \Gamma \to U(k)\) and elements of \(H_2(\Gamma, \mathbb{Z})\) was introduced in \cite{arXiv:2204.10354} using the bar-resolution definition of \(H_2(\Gamma, \mathbb{Z})\).
  Let $c=\sum_{j=1}^{m}k_j [x_j|y_j]\in Z_2(\G,\Z)$ be a $2$-cycle. Let \(\rho: \Gamma \to U(k)\)  be an $(S,\delta)$-representation such that $x_j,y_j\in S$ and $0<\delta<1.$ Then 
 \begin{equation}\label{eq:111f}
\openp \rho ,c\closep=\frac{1}{2\pi i}\sum_{j=1}^{m} k_j\mathrm{Tr}\left( \log \left(\rho(x_j)\rho(y_j)\rho(x_jy_j)^{-1}\right)\right) \end{equation}
{
\begin{remark}
If one replaces $r$ and $c$ by homologous 2-cycles $r'$ and $c'$ the pairings yield the same values provided that $S$ is sufficiently large and $\delta$ is sufficiently small: \(\pmb{\langle} \rho, r'\pmb{\rangle}=\pmb{\langle} \rho, r' \pmb{\rangle}\), and  $\openp \rho,c\closep=\openp \rho,c'\closep.$
While the second equality is proved directly in~\cite{arXiv:2204.10354}, we think that it is more conceptual to point out that it follows from Proposition~\ref{prop:2cocycle} by using the (local) cocycle condition.
On the other hand, the first equality follows from Theorem~\ref{thm:ddd}.
\end{remark}
}
We have the following consequence of \cite[Prop. 4.1]{partA}:
\begin{proposition}\label{prop:2.3+}
Let $\Gamma$ be a countable discrete group and let $\varphi:H_2(\Gamma,\mathbb{Z})^{Hopf} \to H_2(\Gamma,\mathbb{Z})$ be the canonical isomorphism. 
Let $[r]\in H_2(\G,\Z)$ be represented by $r=\prod_{i=1}^g[a_i,b_i]$ and let $\varphi([r])$ be represented by a 2-cycle   $\varphi(r)=\sum_{j=1}^{m}k_j [x_j|y_j]$ in the bar resolution.
There exist  a finite set $S\subset\Gamma$ and $\delta>0$, such that for any  $(S,\delta)$-representation $\rho:\G \to U(k)$ we have
\[\pmb{\langle} \rho, r\pmb{\rangle}=\openp \rho,\varphi(r)\closep.\]
\end{proposition}

\begin{proof}
This can be established by applying \cite[Proposition 4.1]{partA} and proceeding via a proof by contradiction. Nevertheless, we rewrite the proof from \cite{partA} to make the computation of an explicit $S$ and $\delta$ possible.

By \cite[chapter II.5 Exercise 4]{iBrown:book-cohomology}, 
 if $r\in H_2(\G,\Z)$ is represented by $\prod_{i=1}^g[a_i,b_i]$ in the Hopf formula,
then a 2-cycle representative for the class of $\varphi(r)$ is the element $\sum_{i=1}^g d_i$,  where 
\begin{equation}\label{eq:hoppy}
    d_i=[I_{i-1}|\bar a_i]+[I_{i-1}\bar a_i|\bar b_i]-[I_{i-1}\bar a_i\bar b_i\bar a_i^{-1}|\bar a_i]-[I_i|\bar b_i]
\end{equation}
and $I_i=[\bar a_1,\bar b_1]\cdots[\bar a_i,\bar b_i]$. Let $[\varphi(r)]=\sum_{i=1}^g [d_i]$ with $d_i$ as in \eqref{eq:hoppy}. Let $L(m)=\frac{1}{2\pi i}\Tr(\log(m))$ and note that if $\rho$ is sufficiently multiplicative on a sufficiently large finite subset of $\G$ we may use Lemma~\ref{loghom} to compute
\begin{align*}
\langle[\rho],[d_i]\rangle=&L(\rho(I_{i-1})\rho(\bar a_i)\rho(I_{i-1}\bar a_i)^{-1})
+L(\rho(I_{i-1}\bar a_i)\rho(\bar b_i)\rho(I_{i-1}\bar a_i\bar b_i)^{-1})\\
&+L(\rho(I_{i-1}\bar a_i\bar b_i)\rho(\bar a_i)^{-1}\rho(I_{i-1}\bar a_i\bar b_i\bar a_i^{-1})^{-1})+L(\rho(I_i\bar b_i)\rho(\bar b_i)^{-1}\rho(I_i)^{-1})\\
=&L(\rho(I_{i-1})\rho(\bar a_i)\rho(\bar b_i)\rho(\bar a_i)^{-1}\rho(I_i{\bar b_i})^{-1}\rho(I_i\bar b_i)\rho(\bar b_i)^{-1}\rho(I_i)^{-1}) 
\\
=&L(\rho(I_{i-1})[\rho(\bar a_i),\rho(\bar b_i)]\rho(I_i)^{-1}).
\end{align*}
Consequently,
\[
\langle[\rho],[\varphi(r)]\rangle=\sum_{i=1}^g \,L(\rho(I_{i-1})[\rho(\bar a_i),\rho(\bar b_i)]\rho(I_i)^{-1})=L\left(\rho(I_1)\Big(\prod_{i=1}^g[\rho(\bar a_i),\rho(\bar b_i)]\Big)\rho(I_g)^{-1}\right)
.\]
Since $I_g=I_1=1$,  we obtain the desired conclusion.
\end{proof}
 \begin{corollary}\label{indexx}
 	For any $c\in Z_2(\G,\Z)$, there exist a finite set $S \subset \G$  and $\delta>0$ such that if $\rho:\G \to U(k)$ is an $(S,\delta)$-representation, then
 	 \begin{equation}\label{eq:indexc01+}
 		\rho_{\sharp}( \alpha^\G([c])) =\langle [\omega],[c]\rangle.
 	\end{equation}
 \end{corollary}
 \begin{proof}
 	The cycle $c$ is supported by some finite subcomplex $Y$ of $B\G$. Having $Y$ fixed, by Proposition~\ref{well-defined} there are suitable $S$ and $\delta$ so that  $\omega\in Z^2(\G,\R)_Y$ is well-defined and $\langle [\omega],[c]\rangle$ is meaningful as discussed in Lemma~\ref{2cy-cocy}. By equation~\eqref{eq:111f} and definition of $\omega$,   the pairing 
 	$\openp \rho, c\closep$ coincides with the following Kronecker pairing: $$\openp \rho, c\closep=\langle [\omega],[c]\rangle.$$ 
 	Substituting $c=\varphi(r),$  we rewrite this as  $\openp \rho, \varphi(r)\closep=\langle \omega,\varphi(r)\rangle.$ By
 	Proposition~\ref{prop:2.3+}, we deduce that $\pmb{\langle} \rho, r\pmb{\rangle}=\langle \omega,\varphi(r)\rangle.$
 	The desired conclusion follows by applying Theorem~\ref{thm:ddd}, according to which
 	$ \rho_\sharp( \alpha^\G(\varphi(r)))=\pmb{\langle} \rho, r\pmb{\rangle}.$
 \end{proof}
 	\section{ Proof of the first part of Theorem \ref{thm:c11}}\label{Proof1}
 Consider the Mishchenko line-bundle $\ell_\G$ with fiber $\ell^1(\G)$ defined by
${E\G}\times_\G C^*(\G)\to B\G$, where $\G\subset \ell^1(\G)$ acts diagonally.
  Let $Y$ be a  finite subcomplex $Y\subset B\G$. 
Let  $(U_i)_{i\in I}$ be a finite covering of $Y$ by open sets
such that $\ell$ is trivial on each $U_i$ and $U_i\cap U_j$ is connected. Using  trivializations
of restrictions of $\ell$  to $U_i$ one obtains group elements $s_{ij}\in G$
 which define a 1-cocycle that is constant on each nonempty set $U_i\cap U_j$  and which represents $\ell_Y$, the restriction of $\ell$ to $Y$.  Thus $s_{ij}^{-1}=s_{ji}$ and $s_{ij} \cdot s_{jk}=s_{ik}$ whenever $U_i\cap U_j \cap U_k \neq \emptyset.$
 Let $(\chi_i)_{i\in I}$
be positive continuous functions with $\chi_i$ supported in $U_i$ and such that
$\sum_{i\in I} \chi^2_i=1$. Set $m=|I|$ and let $(e_{ij})$ be the canonical matrix unit of $M_m(\C)$. Then $\ell_Y$ is represented by the selfadjoint projection
\[e_Y= \sum_{i,j\in I}  \chi_i\chi_j\otimes s_{ij}\otimes e_{ij}\in C(Y)\otimes \ell^1(\G)\otimes M_{m}(\C).\]

\begin{definition}\label{def:pushf}
    For a unital map $\rho:\G \to U(k),$ we denote again by $\rho$ its contractive linear extension $\ell^1(\G)\to M_k(\C)$.
Consider a finite set $S \subset \G$ such that $\{s_{ij}:i,j \in I\}\subset S$  and $\delta>0$.  
Suppose that $\|\rho(st)-\rho(s)\rho(t)\|<\delta$ for all $s,t \in S.$
Then the element $$h_Y=(\id{} \otimes \rho \otimes \id{} )(e_Y)=\sum_{i,j\in I}  \chi_i\chi_j\otimes \rho(s_{ij})\otimes e_{ij}\in C(Y)\otimes M_{k}(\C)\otimes M_{m}(\C),$$
is almost a projection.
Using functional calculus, if $\delta$ is sufficiently small,  one finds an projection $e_Y \in C(Y)\otimes M_{k}(\C)\otimes M_{m}(\C)$ close to $h_Y$.  We denote by {$E^Y_\rho$} the  vector bundle representing $e_Y$ and write $E^Y_\rho$=$E_\rho$ if $Y=B\G$ is compact.
\end{definition}
In Proposition~\ref{near-cocycle} we verify that $E^Y_\rho$ is given by a cocycle $u_{ij}:U_i\cap U_j \to U(V)$  such that $u_{ij}$ is a small perturbation of $\rho(s_{ij})$
 
 \begin{theorem}\label{thm:c1} 
 Let $\G$ be a countable discrete group. For any finite subcomplex $Y$ of $B\G$,
		there are  a finite set $S\subset \G$ and $\delta>0$ such that for any $(S,\delta)$-representation $\rho:\G \to U(n),$ the vector bundle $E^Y_\rho$ is well-defined and
		
		(1)  The equation
\begin{equation}\label{cocoa}
  \omega(a,b):=\frac{1}{2\pi i}\Tr(\log(\rho(a)\rho(b)\rho(ab)^{-1})),
\end{equation}
		defines a local 2-cocycle $\omega \in Z^2(\G,\R)_{Y}$  with the property that
		$c_1(E^Y_\rho)=[\omega]$ in $H^2(Y,\R).$ {Moreover, for any integral 2-cycle  $c\in Z_2(Y,\Z),$ the corresponding Kronecker pairing takes integral values, $\langle[\omega], [c]\rangle\in \Z$.}
		
  If $\rho$ is a projective representation $\rho:\G \to U(n)$ with cocycle $\lambda(a,b)=e^{2\pi i\sigma(a,b)},$ for $\sigma\in Z^2(\G,\Q)$ and $|\sigma(a,b)|<\delta$ for $a,b\in S$, then $\omega(a,b)=n\sigma(a,b)$ for all $a,b\in S$, hence 
  	$c_1(E^Y_\rho)=n{i^*[\sigma]},$ where $i^*:H^2(B\G,\Q)\to H^2(Y,\Q)$ is the map induced by the inclusion $i:Y \to B\G.$
\end{theorem}
\begin{proof} By the universal coefficient theorem it suffices to show that 
	$$
	\langle c_{1}(E^Y_{\rho}), [c]\rangle=\langle [\omega], [c]\rangle,
	$$
	for all $c \in Z_{2}(Y,\Z)\subset Z_{2}(\G,\Z).$
	Recall the isomorphism $\varphi:H_2(\G,\Z)^{Hopf} \to H_2(\G,\Z) $ discussed earlier. We can write $[c]=\varphi(r)$ for $r\in H_2(\G,\Z)^{Hopf}$. Then $\varphi(r)\in H^2(\G,\Z)\cong  H_2(B\G,\Z).$
	By Corollary~\ref{indexx}  we have 
	\begin{equation}
		\rho_\sharp (\alpha^\G(\varphi(r)))=\langle [\omega], \varphi(r)\rangle,
	\end{equation}
	so that it suffices  to show that 
	\begin{equation}
		\rho_\sharp (\alpha^\G(\varphi(r)))=\langle c_{1}(E^Y_{\rho}), \varphi(r)\rangle.
	\end{equation}
	By definition $\alpha^\G=\mu_1^\G  \circ \beta^\G$.
	By  \cite[Thm. 3.2]{BB}, for any $z\in K_0(Y)$ and fixed projections $q_0,q_1$ in matrices over $\ell^1(\G)$ with $\mu_1(z)=[q_0]-[q_1],$ for all $S\subset \G$  sufficiently large and all $\delta$  sufficiently small, then  map $\rho:\ell^1(\G)\to \C$ satisfies
	\begin{equation}\label{eqref:indexxx}
		\rho_\sharp(\mu_1(z))=\langle [E^Y_\rho],z \rangle.
	\end{equation}
	where the right hand is given by the pairing $K^0(Y)\times K_0(Y)\to \Z$.
	We have a commutative diagram with vertical maps the Chern characters
	\[\begin{tikzcd}
		K^0(Y)\times K_0(Y)\arrow[ shift left=8 ]{d}{\ch} \arrow[shift right=8]{d}{\ch} \arrow[r] & \Z\ar[d] \\
		H^{ev}(Y,\Q)\times H_{ev}(Y,\Q) \arrow[r] & \Q
	\end{tikzcd}
	\]
	It follows from ~\eqref{eqref:indexxx}
	\[\rho_\sharp(\mu_1(z))=\langle \ch(E^Y_\rho),\ch(z)\rangle\]
	If we set  $z=\beta^Y(\varphi(r)),$ then by \eqref{eqn:inverse}
	$$(\ch\otimes \id{\Q}) (\beta^Y(\varphi(r))=\varphi(r)_\Q\in H_2(Y,\Q)$$ 
		and hence
	$$
	\rho_\sharp (\alpha^\G(\varphi(r)))=\rho_{\#}(\mu_1(\beta^\G(\varphi(r))))=\langle c_{1}(E^Y_{\rho}), \varphi(r)\rangle,
	$$
	as desired.
	{The left hand side is an integer and $\langle c_{1}(E_{\rho}), [c]\rangle$ is the image in $\Q$ of  the integer $\langle c_{1}(E_{\rho})_{\Z}, [c]_{\Z}\rangle$  given by the pairing $H^2(Y,\Z)\times H_2(Y,\Z)\to \Z$.}	
	
	For the second part of the statement, note that if $\rho$ is a projective representation with cocycle $\lambda(a,b)=e^{2\pi i\sigma(a,b)},$ for $\sigma\in Z_2(\G,\Q),$
	then 
	\[\omega(a,b):=\frac{1}{2\pi i}\Tr(\log(e^{2\pi i\sigma(a,b)})1_n)=n\sigma(a,b).\]
\end{proof}
 \begin{corollary}\label{chernofprojective}
 Let $x\in H^2(\G,\Z)$ be represented by a 2-cocycle $\sigma\in Z^2(\G,\Z)$. For any finite subcomplex $Y$ of $B\G$,
there is $n_0\in \N$ such that  if a unital map $\rho:\G \to U(k)$ satisfies the equation
  $$\rho(a)\rho(b)\rho(ab)^{-1}=e^{\frac{2\pi i}{n}\sigma(a,b)}1_k,\quad \forall a,b \in \G$$ for some $n\geq n_0$,
   then  $$c_{1}(E_{\rho})=
  \frac{k}{n} i^*(x)\in H^2(Y,\Q).$$
\end{corollary}

\begin{proposition}\label{prop:pullback}
  Suppose that $f:\G \to \La$ is a homomorphism of groups and that both $B\G$ and $B\La$ are finite $CW$-complexes.
  If $\rho:\La \to U(k)$ is sufficiently multiplicative, then $[f^*(E_\rho)]=[E_{\rho\circ f}]$ in $K^0(B\G)\otimes \Q$.
\end{proposition}

\begin{proof}
	By the universal coefficient theorem it suffices to check that
	\[\langle [f^*(E_\rho)], z\rangle=\langle[E_{\rho\circ f}], z\rangle,\quad \forall z \in K_0(B\G).\]
	Equivalently
	\[\langle [E_\rho], f_*(z)\rangle=\langle[E_{\rho\circ f}], z\rangle,\quad \forall z \in K_0(B\G),\]
 where $f_*:K_0(B\G)\to K_0(B\Lambda)$ is the map on K-homology induced by $f$.
	Using \eqref{eqref:indexxx} 
this is equivalent to 
	\begin{equation}\label{eqref:inde}
		\rho_\sharp(\mu_1 (f_*(z))=(\rho \circ f)_\sharp(z)
	\end{equation}
Since $\mu_1$ is a natural transformation, $\mu_1 (f_*(z))=f_*(\mu_1 (z))$ and hence
\[\rho_\sharp(\mu_1 (f_*(z))=\rho_\sharp(f_*(\mu_1 (z))) =(\rho \circ f)_\sharp(z).\]
\end{proof}
For a more general version of Proposition~\ref{prop:pullback} see Proposition~\ref{prop:pullback1}
	\section{ Proof of the second part of Theorem \ref{thm:c11}}

        In this section prove the second part of Theorem \ref{thm:c11}, see Theorem~\ref{thm:c12}.
   The proof is obtained by putting together  Propositions~\ref{near-cocycle}, ~\ref{prop:projB} and ~\ref{bundles}, which we believe to be of independent interest, with Theorem~\ref{thm:c1}.

Let $A$ be a unital $C^*$-algebra and let $V$ be a finitely generated (projective) right Hilbert $A$-module.  Let $\mathcal{L}(V)$ be the  $C^*$-algebra of adjointable $A$-linear operators acting on $V$. Let $X$ be a compact Hausdorff space and let $\mathcal{U}=(U_i)_{i\in I}$  be a finite open  cover of $X.$  Let $(\chi_i)_{i\in I}$
be positive continuous functions with $\chi_i$ supported in $U_i$ and such that
$\sum_{i\in I} \chi^2_i=1$. Set $m=|I|$ and let $(e_{ij})$ be the canonical matrix unit of $M_m(\C)=\mathcal{L}(\C^m)$.
Let $e_i:\C \to \C^m$ be the inclusion on the $i^{th}$-component and let $e_i^*: \C^m \to \C$ be its adjoint. Thus $e_i^*\circ e_j=e_{ij}$ and
$e_i\circ e_j^*=\delta_{i,j}\mathrm{id}_\C.$

The proof of the following proposition, which perturbs an approximate 1-cocycle to a true cocycle, is based on an idea of Kubota from \cite[Lemma~4.4]{Kubota3}. A different but more involved proof is due to Phillips and Stone \cite{MR832541}, \cite{MR1124246}, see also \cite{AA} and \cite{Carrion-Dadarlat}.

We use the following two elementary perturbation properties.
Let $h$ be a self-adjoint element in a unital $C^*$-algebra $D$ with $\|h\|\leq 1$ and $\|h^2 - h\| < \varepsilon < 2/9$. By functional calculus, there exists a self-adjoint projection $p\in D$ such that:
\begin{equation}\label{a1}
\|p-h\| < {3\varepsilon}/{2}
\end{equation}
If $p$ is a self-adjoint projection in $D$ and  $w\in D$ is a contraction with $\|w^*w - 1\| < \varepsilon$ and $\|ww^* - p\| < \varepsilon < 1/15$, then by functional calculus, there is a partial isometry $u\in D$ such that 
\begin{equation}\label{a2}
\|u-w\| < 12{\varepsilon}, \quad \text{and} \quad u^*u=1, \,\, uu^*=p.
\end{equation}
{\begin{proposition}\label{near-cocycle}
Suppose that
$v_{ij}:U_i\cap U_j \to U(V)$ is continuous, $v_{ij}=v_{ji}^{*},$ $v_{ii}=1$ and
\begin{equation}\label{almost-cocycle}
	\max_{i,j\in I} \sup_{x\in U_i\cap U_j \cap U_k} \|v_{ij}(x)v_{jk}(x)-v_{ij}(x)\|<\delta<1/40m^2.
 \end{equation} 	Then there exists a continuous 1-cocycle  $u_{ij}:U_i\cap U_j \to U(V),$ 
 $u_{ij}(x)u_{jk}(x)=u_{ik}(x)$ for all $x \in U_i\cap U_j\cap U_k,$ such that
\[\|u_{ij}(x)-v_{ij}(x)\|<66 m^2 \delta, \quad \forall x \in U_i\cap U_j.\]
\end{proposition}
\begin{proof}
Consider the selfadjoint contractive element
\begin{equation}\label{a3}
h= \sum_{i,j\in I}  \chi_i\chi_j\otimes v_{ij}\otimes e_{ij}.
\end{equation}
 
  of the $C^*$-algebra $ C(X)\otimes \mathcal{L}(V)\otimes M_{m}(\C)$.
Then
\[h^2-h=\sum_{i,k} \left(\sum_j e_{ik}\otimes \chi_i\chi_k \chi_j^2 \otimes (v_{ij}v_{jk}-v_{ik})\right)\]
and hence $\|h^2-h\|<m^2\delta<1/10.$ It follows from \eqref{a1} that there is a self-adjoint projection $p\in C(X)\otimes \mathcal{L}(V)\otimes M_{m}(\C)$ with
\begin{equation}\label{a4}
	\|p-h\|<{3 m^2\delta}/{2}.
	\end{equation}
For each $i\in I$ and $x\in U_i$ define the isometry $v_i(x): V \to V\otimes \C^m$ by
\[v_i(x)=\sum_{r} \chi_r(x)\otimes v_{ri}(x)\otimes e_r.\]
Its range projection is denoted $p_i=v_iv_i^*\in C_b(U_i)\otimes \mathcal{L}(V)\otimes M_{m}(\C),$ and
\begin{equation}\label{a5}
p_i=\sum_{r,s} \chi_r\chi_s \otimes v_{ri}v_{is}\otimes e_{rs}.
\end{equation}
If we denote the restriction of $h$ to $U_i$ by $h$ as well, we obtain from \eqref{almost-cocycle} and \eqref{a3} that
\[\|h-p_i\|<m^2\delta.\]
Consequently,
\begin{equation}\label{a6}
\|p-p_i\|\leq \|p-h\|+\|h-p_i\|<{3 m^2\delta}/{2}+m^2\delta={5 m^2\delta}/{2}.
\end{equation}
Next perturb each $v_i$ to an isometry $u_i$ with range projection $p|_{U_i}$ as follows.
We denote the restriction of $p$ to $U_i$ by $p$ as well.
 Let $w_i=pv_i$. Then 
\[\|v_i-w_i\|=\|p_iv_i-pv_i\|< {5 m^2\delta}/{2}\]
\[\|w_i^*w_i-1\|=\|v_i^*(p-p_i)v_i\|<{5 m^2\delta}/{2}\]
\[\|w_iw_i^*-p\|=\|pv_iv_i^*p-p\|=\|p(p_i-p)p\|<{5 m^2\delta}/{2}.\]
Since ${5 m^2\delta}/{2}<1/15$ by hypothesis, it follows from \eqref{a2} that there is a isometry  $u_i$ such that
\[\|u_i-w_i\|< 12\cdot \frac{5}{2}m^2\delta= {30 m^2\delta},\] 
and
\begin{equation}\label{a7}
	u_i^*u_i=1,\quad u_iu_i^*=p|_{U_i}.
\end{equation}
Since
\begin{equation}\label{a9}
	v_i^*v_j=\sum_r \chi_r^2\otimes v_{ir}v_{rj} \otimes 1
\end{equation}
we obtain from \eqref{almost-cocycle} and \eqref{a9}
\begin{equation}\label{a10}
\|v_{ij}(x)-v_i(x)^*v_j(x)\|<\delta,\quad \forall x \in U_i\cap U_j.
\end{equation}
Then
\[\|u_i-v_i\|\leq \|u_i-w_i\|+\|w_i-v_i\|\leq {30 m^2\delta}+ {5 m^2\delta}/{2}={65 m^2\delta}/{2}\]
and hence
\begin{equation}\label{a111}
	\|u_i(x)^*u_j(x)-v_i(x)^*v_j(x)\|<{65 m^2\delta},\quad \forall x \in U_i\cap U_j.
\end{equation}
Let $u_{ij}:U_i\cap U_j \to U(V)$ be defined by
\[u_{ij}(x)=u_i(x)^*u_j(x)\]
From \eqref{a10} and \eqref{a111}
\[\|u_{ij}(x)-v_{ij}(x)\|< {65 m^2\delta}+\delta \leq 66 m^2 \delta, \quad \forall x \in U_i\cap U_j.\]
Moreover
\[u_{ij}(x)u_{jk}(x)=u_{ik}(x),  \quad \forall x \in U_i\cap U_j\cap U_k, \] since
\[u_{i}(x)^*u_{j}(x)\cdot u_{j}(x)^*u_{k}(x)=u_{i}(x)^*p(x)u_{k}(x)=u_{i}(x)^*\left(u_{i}(x)u_{i}(x)^*\right)u_{k}(x)=u_{i}(x)^*u_{k}(x).\]
\end{proof}
}
For $G$ a topological group, a principal $G$-bundle is flat 
if it has a set of trivializations with constant transition functions. Equivalently,
the bundle has a reduction to the group $G_d$, the underlying discrete group of $G$, \cite{Dupont}. 
\begin{proposition}\label{prop:projB}
	Let $E$ be a locally trivial hermitian vector bundle of rank $n$ on a paracompact Hausdorff space  $X$ such that the associated projective bundle $P(E)$ is isomorphic to a flat $PU(n)$-bundle.
	Then the Chern classes of $E$ are $c_k(E)=\frac{1}{n^k}{n\choose k}c_1(E)^k$ and  the Chern character is $\ch(E)=ne^{\frac{1}{n}c_1(E)}.$
\end{proposition}

We thank Rufus Willet for alerting us that, in the case where $X$ is a manifold, this corresponds to~\cite[Proposition 2.3.1]{diffofbundles}.

\begin{proof} 
	Consider the commutative diagram: 
	\begin{equation}\label{unitarygroups}	\begin{tikzcd}
			1 \arrow[r] & \mathbb{Z}/n \arrow[r] \arrow[d, hook] & SU(n) \arrow[r] \arrow[d, hook, "j"] & PSU(n) \arrow[r] \arrow[d, equal] & 1 \\
			1 \arrow[r] & U(1) \arrow[r,"d"] & U(n) \arrow[r, "\pi"] & PU(n) \arrow[r] & 1
		\end{tikzcd}
	\end{equation}
and the corresponding fibration of classifying spaces
\[
\begin{tikzcd}
 B\mathbb{Z}/n \arrow[r]  & BSU(n) \arrow[r, "\pi\circ j"]  & BPU(n)   
\end{tikzcd}
\]
Since the rational cohomology of $\mathbb{Z}/n$ vanishes in positive degrees, the $E_2$-page of the Serre spectral sequence of this fibration is concentrated on the $q=0$ row. It follows that the edge homomorphism
 $$(\pi\circ j)^*:H^*(BPU(n);\Q)\to H^*(BSU(n);\Q)$$ 
 is an isomorphism of groups. 
 Consider the commutative diagram
 \begin{equation}\label{unitarygroupss}
 \begin{tikzcd}
 	H^*(X;\Q)  & H^*(BU(n);\Q)\arrow[d, "j^*"]\arrow[l, "d^*"'] \arrow[r]  &H^*(BPU(n);\Q) \arrow[l, "\pi^*"'] \arrow[dl, "\cong" 
 	]    \\
 	& H^*(BSU(n);\Q) &    
 \end{tikzcd}
	\end{equation}
 
 Recall that by \cite{Toda}, the map 
 \[
 j^*: H^*(BU(n); \Q) \to H^*(BSU(n);\Q) 
 \]
 is the canonical projection of graded algebras
 \[\mathbb{Q}[c_1,c_2, c_3, \ldots, c_n]\to \mathbb{Q}[c_2, c_3, \ldots, c_n]\]
 which maps $c_1$ to $0$  and $c_k$ to $c_k$ if $k>1.$ Here $c_k$   are the universal Chern classes, \( \deg(c_k) = 2k \). 
 
 It follows from the diagram \eqref{unitarygroupss} that for any $x\in H^{2k}(BU(n);\Q)$ there are $y\in H^{2k}(BPU(n);\Q)$ and $e\in\ker(j^*)\cap H^{2k}(BU(n);\Q) =c_1H^{2k-2}(BU(n);\Q)$ such that $x=\pi^*(y)+e$. By induction on $k$ we can thus see that 
 	$$x=\pi^*(y_1)+c_1\pi^*(y_2)+\cdots+c_1^{k-1}\pi^*(y_{k-1})+\lambda c_1^k,$$
 	for some $\lambda\in \Q$. In particular
 	\begin{equation}\label{ck}
 		c_k=\pi^*(y_1)+c_1\pi^*(y_2)+\cdots+c_1^{k-1}\pi^*(y_{k-1})+\lambda_k c_1^k,
 	\end{equation}
 	 for a unique constant $\lambda_k\in \Q.$
 	 Since the map $\pi \circ d$ is constant, it follows that $d^*\circ\pi^*=0$ in all degrees other than zero.
 	 Thus from equation~\eqref{ck} we deduce that
 	 \begin{equation}\label{aaa}
 	 	d^*(c_k)=\lambda_kd^*(c_1)^k.
 	 \end{equation}
 
 	Since $c_k$ is the $k$th Chern class of the tautological vector bundle $W$ over $BU(n)$, and $H^*(BU(1);\Q))=\Q[c_1]$ where $c_1$ is the first Chern class of the tautological line bundle, one sees that $d^*(W)$ is the $ n$-fold direct sum of the tautological line bundle over $BU(1)$. By the Whitney sum formula for the total Chern class, it follows that 
 	$$d^*(c_k)={n\choose k}c_1^k\quad \text{and in particular}\quad d^*(c_1)=nc_1.$$
 	In conjunction with equation~\eqref{aaa} we deduce that $\lambda_k=\frac{1}{n^k}{n\choose k}$.
 	
 
 If $f:X \to BU(n)$ is the classifying map of $E$ we have a  diagram
\[
\begin{tikzcd}
	H^*(X;\Q)  & H^*(BU(n);\Q)\arrow[l, "f^*"'] \arrow[r]  &H^*(BPU(n);\Q) \arrow[l, "\pi^*"']    
\end{tikzcd}
\]
If $G$ is a compact Lie group, then by \cite[p.145]{Dupont} the canonical inclusion $G_d \to G$ induces the zero map $H^q(BG;\Q)\to H^q(BG_d;\Q),$ $q>0.$ The assumption that $P(E)$ is flat means that the associated principal $PU(n)$-bundle is flat. It follows that  the classifying map of $P(E),$ namely $\pi\circ f: X \to BPU(n)$ factors up to homotopy through the map $BPU(n)_d\to BPU(n)$ and hence $f^*\circ \pi^*=0$ in all dimensions $k>0$.
Thus from equation~\eqref{ck} we deduce that
\begin{equation}\label{aaaa}
	f^*(c_k)=\lambda_kf^*(c_1)^k.
\end{equation}
Consequently,
\begin{equation}\label{zzz}
c_k(E)=\lambda_kc_1(E)^k=\frac{1}{n^k}{n\choose k}c_1(E)^k.
\end{equation}
The formula for the Chern character is proved via the splitting principle.
By this principle, there exists a space $Y$ and a map $p: Y \to X$ such that $p^*: H^*(X, \mathbb{Q}) \to H^*(Y, \mathbb{Q})$ is injective, and the pullback bundle splits as
$ p^* E \cong L_1 \oplus \cdots \oplus L_n,$
where $L_i$ are line bundles on $Y.$ 
We will show that $p^*\ch(E)=p^*(ne^{\frac{1}{n}c_1(E)})$ and this will imply the equality 
$\ch(E)=ne^{\frac{1}{n}c_1(E)}$ by injectivity of $p^*$.

 Let $x_i = c_1(L_i)$ and set $z=\frac{1}{n}(x_1+\cdots+x_n) \in H^2(Y, \mathbb{Q})$.
 The total Chern class is $c(p^*E)=(1+x_1)\cdots(1+x_n),$  and hence $c_1(p^*(E))=x_1+\cdots+x_n,$
 and more generally $c_k(p^*(E))=e_k(x_1, \ldots, x_n),$ where is the $k$-th elementary symmetric polynomial. On the other hand
 by \eqref{zzz} we have
\begin{equation*}
c_k(p^* E)  = p^* c_k(E)=\frac{1}{n^k}{n\choose k}c_1(p^*E)^k={n\choose k}\left(\frac{x_1+\cdots+x_n}{n}\right)^k={n\choose k}z^k.
\end{equation*} This shows that
$e_k(x_1, \ldots, x_n)=e_k(y_1, \ldots, y_n),$
where $y_1=\cdots=y_n=z$.
 Since 
  $ e^{x_1}+\cdots+e^{x_n}$ is a symmetric function, it follows that
 \[p^*\ch(E)=\ch(p^* E)=e^{x_1}+\cdots+e^{x_n}=e^{y_1}+\cdots+e^{y_n}=ne^z\]
We conclude the proof by observing that $$p^*(ne^{\frac{1}{n}c_1(E)})= ne^{\frac{1}{n}c_1(p^*E)}=ne^{\frac{1}{n}(x_1+\cdots+x_n)}=ne^z.$$
\end{proof}
Let $B$ be a C$^*$-algebra.  Let \( X \) be a compact metric space, and let \( (U_i)_{i \in I} \) be a finite open cover of \( X \). Let $\alpha=(\alpha_{ij})$ and $\beta=(\beta_{ij})$ be two 1-cocycles
\[
\alpha_{ij},\beta_{ij} : U_i \cap U_j \to \Aut(B).
\]
Let $A_\alpha$ and $A_\beta$ be the $C(X)$-C$^*$-algebras of continuous sections in the  locally trivial C$^*$-bundles with fiber $B$ associated to $\alpha$ and $\beta$.
Define \[\mathrm{dist}(\alpha,\beta):=\max_{i,j\in I}\sup_{x\in U_i\cap U_j}\|\alpha_{ij}(x)-\beta_{ij}(x)\|\]
\begin{proposition}\label{bundles}
	Let $B$ be a separable and nuclear $C^*$-algebra. If $\mathrm{dist}(\alpha,\beta)<10^{-11},$
	then $A_\alpha$ is isomorphic to $A_\beta$ as $C(X)$-C$^*$-algebras. Thus, the $\Aut(B)$-principal bundles constructed from the cocycles $(\alpha_{ij})$ and $(\beta_{ij})$ are isomorphic.
\end{proposition}
\begin{proof}
  {Let $B$ act faithfully on a separable Hilbert $H$ space such that $BH$ is dense in $H$,  and let $N$ be the von Neumann algebra $N=B''\subset B(H)$. 	 
	Let $Y$ be a countable set dense in $X$ and set $Y_i=U_i\cap Y.$
	Consider the von Neumann algebra
	\[M=\bigoplus_{i\in I} \ell^\infty(Y_i,N)=\bigoplus_{i\in I}M_i\subset B(K)\] acting on the separable Hilbert space  $K=\bigoplus_{i\in I}\ell^2(Y_i)\otimes H$. Let $D$ be the $C^*$-algebra
	\[D=\bigoplus_{i\in I} C_b(U_i,B)=\bigoplus_{i\in I}D_i\] and let
	$j:D \to M$ be the canonical embedding map induced componentwise by restriction of functions and the inclusion $B\subseteq N$. We view both $D$ and $M$ as $C(X)$-C$^*$-algebras. If $f\in C(X),$ $d\in D$ and $m\in M$ then, $(fd)_i=f|_{U_i}d_i$ and 
	$(fm)_i=f|_{Y_i}m_i.$
	Then $j$ is a $C(X)$-linear map.
	
	Let $(\chi_i)_{i\in I}$ be a partition of unity subordinated to the cover \( (U_i)_{i \in I} \).
	Define a  map $\Phi_\alpha: D \to D,$ where for $d=(d_i)\in D,$ the components of $\Phi_\alpha(d)$ are given by 
	\[\Phi_\alpha(d)_i(x)=\sum_{j\in I} \chi_j(x)\alpha_{ij}(x)(d_j(x)),\]
	for all $x\in U_i.$ One verifies immediately that $\Phi_\alpha$ is a contractive completely positive linear map.
	Moreover, if $\beta=(\beta_{ij})$ is another 1-cocycle, then 
	\[\|\Phi_\alpha-\Phi_\beta\|\leq \mathrm{dist}(\alpha,\beta):=\max_{i,j\in I}\sup_{x\in U_i\cap U_j}\|\alpha_{ij}(x)-\beta_{ij}(x)\|\]
	Let $A_\alpha$ and $A_\beta$ be the $C(X)$-C$^*$-subalgebras of $D$ defined by 
	\[A_\alpha=\{(d_i)_{i\in I}: d_i(x)=\alpha_{ij}(x)(d_j(x)),\,\, \forall x \in U_i\cap U_j\},\]
	\[A_\beta=\{(d_i)_{i\in I}: d_i(x)=\beta_{ij}(x)(d_j(x)),\,\, \forall x \in U_i\cap U_j\}.\]
	Let us note $\Phi_\alpha(d)=d$ for all $d\in A_\alpha$. Indeed if $d\in A_\alpha$, and $x\in U_i$, then
	\[\Phi_\alpha(d)_i(x)=\sum_{j\in I} \chi_j(x)\alpha_{ij}(x)(d_j(x))=\sum_{j\in I} \chi_j(x)d_i(x)=d_i(x).\]
	By the same argument, the elements of $A_\beta$ are fixed by $\Phi_\beta$. In particular both C$^*$-algebras $A_\alpha$ and $A_\beta$ are nuclear since $D$ is nuclear and the maps $\Phi_\alpha $ and $\Phi_\beta $ are contractive and completely positive.
	Let $a\in A$ and $b\in B$ with $\|a\|,\|b\|\leq 1$. Then
	\[\|a-\Phi_\beta(a)\|=\|\Phi_\alpha(a)-\Phi_\beta(a)\|\leq \mathrm{dist}(\alpha,\beta), \quad \|\Phi_\alpha(b)-b\|\leq \mathrm{dist}(\alpha,\beta). \] 
	Thus the Hausdorff distance between the unit balls of $A_\alpha$ and $A_\beta$ is at most $\mathrm{dist}(\alpha,\beta)$. If $\mathrm{dist}(\alpha,\beta)<\delta:=10^{-11},$ then by Theorem B from \cite{StuartW}, there is a unitary $u\in j(A_\alpha\cup A_\beta)''\subset M''=M$ such that
	$uj(A_\alpha)u^*=j(A_\beta)$.  In other words, there is a $*$-isomorphism $\theta:A_\alpha \to A_\beta$ and a commutative diagram:}
\[
\begin{tikzcd}
	A_\alpha\arrow[r, "j"] \arrow[d,"\theta" swap] &M \arrow[d, "\mathrm{Ad}_u"]\\
	A_\beta\arrow[r, "j"] &M
\end{tikzcd}
\] Since the $*$-homomorphisms $j$ and $\mathrm{Ad}_u$ are $C(X)$-linear and $j$ is an injective map, it follows that $\theta$ is $C(X)$-linear.
\end{proof}
{\begin{theorem}\label{thm:c12} Let $\G$ be a countable discrete group and let $Y\subseteq B\G$ be a finite subcomplex.
        {There are a finite set $S\subset \G$ and $\delta>0$  such that for any   projective representation $\rho:\G \to U(n),$  $\rho(a)\rho(b)\rho(ab)^{-1}=\lambda(a,b)1_n,$ with $\lambda(a,b)\in \mathbb{T}$ for all $a,b\in \G$, such that $|\lambda(a,b)-1|<\delta,$ for all $a,b\in S$,}
  	we have 
    $$c_k(E_\rho^Y)=\frac{1}{n^k}{n\choose k}[\omega]^k\quad \text{and} \quad 
  \ch(E^Y_\rho)=n \, e^{{ \frac{1}{n}[\omega]}}\in H^{\text{even}}(Y,\R),$$
  where $\omega\in Z^2(\G,\R)_Y$ is a local 2-cocycle defined  by $\omega(a,b):=\frac{1}{2\pi i}\Tr(\log(\lambda(a,b))).$
  \end{theorem}
}
\begin{proof} For all $i,j \in I$, choose unitaries $v_{ij}$ which are small perturbations of $\rho(s_{ij})$ and such that $v_{ij}=v_{ji}^{-1}$.
     Let $u_{ij}$ be the cocycle provided by
Proposition~\ref{near-cocycle}. This is the cocycle corresponding to $E_\rho^Y$.
Let $\pi:U(n)\to PU(n)$ be the canonical map.
Then the cocycle $\pi(u_{ij})$ is a small perturbation of the constant cocycle
$\pi(\rho(s_{ij}))$. It follows from Proposition~\ref{bundles} that
$P(E_\rho^Y)$ is isomorphic to the flat $PU(n)$-bundle constructed from the cocycle $\pi(\rho(s_{ij}))$, and hence $c_k(E_\rho^Y)=\frac{1}{n^k}{n\choose k}c_1(E_\rho^Y)^k$ and $\ch(E^Y_\rho)=n \, e^{{ \frac{1}{n}c_1(E^Y_\rho)}}\in H^{\text{even}}(Y,\R),$ by Proposition~\ref{prop:projB}.
As already shown in Theorem~\ref{thm:c1}, $c_1(E^Y_\rho)=[\omega].$
\end{proof}
\begin{corollary} \label{cor:proj-inv}
	If $\rho:\G \to U(n)$ is a projective representation with cocycle $\lambda(a,b)=e^{2\pi i\sigma(a,b)},$ for $\sigma\in Z^2(\G,\Q)$ and $|\sigma(a,b)|<\delta$ for $a,b\in S$, then $\omega(a,b)=n\sigma(a,b)$ for all $a,b\in S$, so that $c_1(E^Y_\rho)=n{i^*[\sigma]}$ and $\ch(E^Y_\rho)= ne^{{i^* [\sigma]}},$
    where $i^*:H^2(B\G,\Q)\to H^2(Y,\Q)$ is the map induced by the inclusion $i:Y \to B\G.$
	\end{corollary}
\section{Classifying almost representations up to stable equivalence}\label{sec:classifi}
Let $B$ be a unital $C^*$-algebra and let $X$ be  a locally compact and $\sigma$-compact space Hausdorff space. We shall make use of various K-theory groups and follow the notation of Kasparov \cite{Kas:inv}.
The representable K-theory of $X$ with coefficients in $B$ is the group
\[RK^j(X;B):= RKK^j(X;\C,B)=\mathcal{R} KK^j(X;C_0(X),C_0(X)\otimes B).\]
If $X$ is compact, then $RK^j(X;B)= K^j(X;B)\cong KK(\C,C(X)\otimes B)\cong K_j(C(X)\otimes B).$ 
The K-homology groups of compact spaces, $K_j(X)=KK(C(X),\C),$ are extended to  non-compact spaces $X$ by defining 
\[RK_j(X)=\varinjlim K_j(Y),\]
where $Y$ runs through the compact subspaces of $X.$ 

Let $\K$ be the C*-algebra of compact operators on an infinite dimensional separable Hilbert space. 
    Let $e_{0}\in \K$ be a rank-one self-adjoint projection. We denote by $\Proj{B \otimes \K}$ the space of all self-adjoint projections in $B \otimes \K$ and by  $\uProj{B \otimes \K}$ its connected component that contains $1_B\otimes e_{0}$. Recall that $\uProj{B \otimes \K}$ is a model for $BU(B),$ the classifying space of the unitary group $U(B),$ see the proof of \cite[Cor.2.9]{DP1}.
Let us note that any continuous map $e:X \to \Proj{B \otimes \K}=\Proj{ \K(H_B)} $
gives rise to a $C_0(X)$-linear $*$-homomorphism $\varphi:C_0(X) \to \K(H_{C_0(X)\otimes B}),$   $\varphi(f)=f\cdot e,$ and hence it defines an element of $\mathcal{R} KK^0(X;C_0(X),C_0(X)\otimes B)=RK^0(X;B).$

For a countable discrete group $\G$, one can a choose a model for the classifying space $B\G$ which is locally compact and $\sigma$-compact,  and in fact a separable locally finite CW complex,
 see  \cite[p.192]{Kas:inv}.
Throughout this section we will considered a fixed  representation of $B\G$ as a countable increasing  union of finite CW complexes, 
$Y_1\subseteq Y_2\subseteq \cdots \subseteq Y_i \subseteq \cdots$ 
\begin{equation}\label{union}
    B\G=\bigcup_n Y_i.
\end{equation}

 The principal $\G$-bundle $E\G \to B\G$
 induces a canonical element of $RK^0(B\G,C^*(\G))$, which we now describe.
      Mishchenko's flat line bundle $L_\G$ is the canonical bundle
      $E\G\times_\G C^*(\G)\to B\G$ induced by the diagonal action of $\G$. Its fibers are isomorphic to the $C^*(\G)$-Hilbert module $C^*(\G)$. 
      If $\imath:\G \to U(C^*(\G))$ is the canonical inclusion, then $L_\G$ is classified by the map $B\imath: B\G \to BU(C^*(\G)).$
   One can describe the classifying map $e$  of $L_\G$ as follows. 
     By \cite[p.57]{Hus:fibre}, there is a locally finite countable open cover $(U_a)_{a\in I}$ of $B\G$ and positive continuous functions $(\chi_a)_{a\in I}$ such that $(\chi^2_a)_{a\in I}$ is a partition of unity subordinated to the cover and such that the covering space $E\G\to B\G$ is trivial on each open set $U_a$. We denote by $s_{ab}:U_a\cap U_b\to \G$ the corresponding locally constant cocycle. Then 
$L_\G$ is classified by the map 
\begin{equation}\label{class-map}
    e=\sum_{a,b\in I}  s_{ab} \otimes \chi_a\chi_b \otimes e_{ab},
\end{equation}
where $(e_{ab})$ are matrix units for $M_\infty(\C)\subset \K.$
Note that $$e:B\G \to  \Proj{\C[\G]\otimes M_\infty(\C)}\cap \uProj{C^*(\G) \otimes \K}.$$

   We denote by $e_Y$ the restriction of $e$ to a compact subspace of $B\G$ and by $e_i$ the restriction of $e$ to $Y_i.$ Note that
   \[e_i=\sum_{a,b\in I(i)}  s_{ab} \otimes \chi_a\chi_b \otimes e_{ab}\]
   for some finite subset $I(i)$ of $I.$
   We view $e_i$ as an element of $\Proj{ \C[\G]\otimes C(Y_i) \otimes M_{I(i)}(\C)}.$ 

{
Consider an asymptotic homomorphism consisting of a sequence of
   unital maps $\varphi_n:\Lambda \to U(k_n)$. This induces 
   a unital $*$-homomorphism }
   \begin{equation}
   \varphi:C^*(\Lambda)\to B=\prod_n M_{k_n}(\C)/\bigoplus_n M_{k_n}(\C).
   \end{equation}
   Let $\ell_\varphi$ denote the bundle on $B\Lambda$ classified by the map
  \begin{tikzcd}
  	B\Lambda\ar[r,"B\varphi"] & BU(B)=\uProj{B \otimes \K}.
  \end{tikzcd} 
  If $f:\Gamma\rightarrow\Lambda$ is a group homomorphism, and $Bf:B\G\to B\Lambda$, then 
  by basic properties of classifying spaces, 
  \begin{equation}\label{functoriality}
   (Bf)^* (\ell_\varphi)\cong \ell_{\varphi\circ f}. 
  \end{equation}
  For each compact subspace $Z$ of $B\Lambda$, by functional calculus we lift the restriction of $B\varphi$ to $Z$ to a map $p:Z \to \uProj{\left(\prod_n M_{k_n}\right) \otimes \K}$. If we denote the components of $p$ by $(p_n)$ 
  it follows from the definition of $E_{\varphi_n}^Z$ that there is $k$ such that the bundle $E_{\varphi_n}^Z$ is given by the projection $p_n$ for all $n\geq k$. Using equation ~\eqref{functoriality}, and the functoriality of the functional calculus, it follows that for any compact subset of $B\G$, there is $k\in \N$ such that 
  \[[E_{\varphi_n\circ f}^Y]= [f^*(E_{\varphi_n}^{f(Y)}]\in K^0(Y),\quad \text{for all } n\geq k.\]
  We restate this property in the following equivalent form.
\begin{proposition}\label{prop:pullback1}
Suppose that $f:\Gamma\rightarrow\Lambda$ is a group homomorphism, and $Y\subseteq B\Gamma$ is a finite sub-complex. Then there exist a finite subset $S\subset \Lambda$ and $\ep>0$ such that for any  $(S,\delta)$-representation $\rho:\Lambda\rightarrow U(n),$ $[E_{\rho\circ f}^Y]=[f^*(E_{\rho}^{f(Y)})]$ in $K^0(Y)$.
\end{proposition}

The element of $RK(B\G;C^*(\G))$ corresponding to $e$ is denoted by $e_\G.$ Kasparov uses the product 
\[RK(B\G;C^*(\G))\times KK^j(C^*(\G),B)\to RK^j(B\G;B) \] to define the co-assembly map
\[\nu:KK^j(C^*(\G),B)\to RK^j(B\G;B) \]
as the cap product with $e_\G.$
If $Y\subseteq B\G$ is compact, we denote by $\nu_Y$ the composite map
\begin{equation}\label{nuy}
\begin{tikzcd}
  	\nu_Y:KK^j(C^*(\G),B)\ar[r,"\nu"] & RK^j(B\G;B)\ar[r] & RK^j(Y;B)
  \end{tikzcd} 
\end{equation}
For each compact subspace $Y\subseteq B\G$, there is a map $\mu_Y:K_j(Y)\to K_j(C^*(\G))$ defined as the cap product with $[e_Y]\in K^0(Y;C^*(\G))$,  the restriction of $e_\G$ to $Y$. The corresponding inductive limit homomorphism is the (full) assembly map
\[\mu:K_j(B\G)\to K_j(C^*(\G)).\]
The associativity of the Kasparov product, shows that $\nu$ and $\mu$ are linked by a duality relation:
\begin{equation}\label{dualitate}
\nu_Y(x)\otimes_{C(Y)}z=x \otimes_{C^*(\G)}\mu(z),\quad x \in K_j(C^*(\G)),\,\, z \in K_j(Y),
\end{equation}
(see \cite[Lemma 6.2]{Kas:inv}).
In the sequel we shall use the notation $\nu_i=\nu_{Y_i}$ and $\mu_i=\mu_{Y_i}$.
\begin{lemma}\label{cddd}
    The following diagram is commutative
\begin{equation}\label{diag-0}
		\begin{tikzcd}
KK^j(C^*(\G),B) \arrow[r, "g", two heads] \arrow[d, "\nu"] & \Hom(K_*(C^*(\G)),K_*(B)) \ \arrow[d, "\mu^*"] & {} 
		\\
RK^j(B\G;B)\arrow[d, "\beta"] \arrow[r, "g'", two heads] & \Hom(K_*(B\G),K_*(B))\arrow[d, "\alpha", tail, two heads]  &
		\\
		 \varprojlim K^j(Y_i;B) \arrow[r, "g''", two heads] & \varprojlim \Hom(K_*(Y_i),K_*(B))  & 
	\end{tikzcd}	
	\end{equation}
The maps $g,g',g''$ are induced by the natural pairings. The maps $\beta, \alpha$ are induced by inclusions $Y_i\subset B\G$. The $\Hom$--groups are graded according to the parity of $j$.
\end{lemma}
\begin{proof}
    The bottom diagram is commutative by naturality of pairings in topological setting, thus $g''\circ \beta=\alpha\circ g'$.
By equation~\eqref{dualitate}, 
 $g''\circ \beta \circ \nu=\alpha \circ \mu^* \circ g.$ It follows that
 \[\alpha\circ g' \circ \nu=\alpha\circ \mu^* \circ g.\]
Since the map $\alpha$ is bijective as it becomes apparent once we describe it as
        \[\Hom(K_*(B\G),K_*(B))=\Hom(\varinjlim K_*(Y_i),K_*(B))\to \varprojlim \Hom(K_*(Y_i),K_*(B))\] we deduce that $g'  \circ \nu=\mu^* \circ g.$
\end{proof}

\begin{proposition}\label{comm-diagss}
Suppose that $C^*(\G)$ satisfies the UCT. For example $\G$ is a countable amenable group.
       Then there is a commutative diagram with exact rows. 
        \begin{equation}\label{diag-1}
		\begin{tikzcd}
		\Ext(K_*(C^*(\G),K_*(B)) \arrow[r, "f", tail] \arrow[d, "\mu_{ext}^*"] & KK^j(C^*(\G),B) \arrow[r, "g", two heads] \arrow[d, "\nu"] & \Hom(K_*(C^*(\G)),K_*(B)) \ \arrow[d, "\mu^*"] & {} 
		\\
		\Ext(K_*(B\G),K_*(B)) \arrow[r, "f'", tail]\arrow[d, "\alpha_{ext}"] &  RK^j(B\G;B)\arrow[d, "\beta"] \arrow[r, "g'", two heads] & \Hom(K_*(B\G),K_*(B))\arrow[d, "\alpha", tail, two heads]  &
		\\
		\varprojlim \Ext(K_*(Y_i),K_*(B)) \arrow[r, "f''", tail] &  \varprojlim K^j(Y_i;B) \arrow[r, "g''", two heads] & \varprojlim \Hom(K_*(Y_i),K_*(B))  & 
	\end{tikzcd}	
	\end{equation}
    The $\Hom$--groups and  the $\Ext$--groups are graded in accord to the parity of $j$.
    \end{proposition}
		\begin{proof} The maps $g,g',g''$ are induced by the natural pairings. The maps $\alpha_{ext},\beta, \alpha$ are induced by inclusions $Y_i\subset B\G$.
        The first row is the UCT of \cite{RosSho:UCT} which holds for amenable groups by \cite{HigKas:BC} and \cite{Tu:BC}. 
        The second row is exact as it represents
        the universal coefficient theorem expressed by Lemma 3.4 of \cite{Kasparov-Skandalis-kk}.
        The exactness of the third row follows from the same lemma and the property that if $G$ is the injective limit of a system $(G_i)$ of finitely generated groups, then  $\varprojlim^1 \text{Ext}(G_i, H)=0$ by \cite[Thm. 6.2]{Schochet-primer}.
    The bottom diagram is commutative by the naturality of the universal coefficient theorem expressed by Lemma 3.4 of \cite{Kasparov-Skandalis-kk}.
We have seen in Lemma~\ref{cddd} that
        the map $\alpha$ is  injective and that
       \begin{equation}\label{dualii}
      g'  \circ \nu=\mu^* \circ g.  
    \end{equation}
  It remains to show that
    \[f'\circ \mu^*_{ext}=\nu \circ f.\]
    For this we will need to use the naturality of \eqref{dualii} while revisiting the proofs of the two universal coefficient theorems represented by the top two rows.

    Consider a geometric injective resolution of $B$ as in \cite{RosSho:UCT}. 
\[
	\begin{tikzcd}
		0 \arrow[r] & A \arrow[r, "h"] & D \arrow[r, ] & SB \arrow[r] & 0
	\end{tikzcd}
\]
    The associated K-theory sequences give injective resolutions for $K_j(B):$ 
\[
	\begin{tikzcd}
		0 \arrow[r] & K_{j+1}(SB) \arrow[r] & K_{j}(A)\arrow[r, "h_*"] & K_{j}(D)\arrow[r] & 0
	\end{tikzcd}
\]
Consider the commutative diagrams that appear in the proof of the UCTs: 
\[
	\begin{tikzcd}
		\Hom(K_*(C^*(\G)),K_*(A)) \arrow[r,"h_*"]& \Hom(K_*(C^*(\G)),K_*(D)) \arrow[r, " "]  & KK^{j+1}(C^*(\G),B) 
        \\
		KK^j(C^*(\G),A)\arrow[u, "g_A\, \cong"]\arrow[d, "\nu_A"] \arrow[r,"h_*"] & KK^j(C^*(\G),D)\arrow[u, "g_D\, \cong" ]\arrow[r, " "] \arrow[d, "\nu_D"] & KK^{j+1}(C^*(\G),B)\arrow[u, "="] \arrow[d, "\nu_B"]
        \\
		RK^j(B\G,A)\arrow[d, "g_A'\, \cong"] \arrow[r,"h_*"] & RK^j(B\G,D) \arrow[r] \arrow[d, "g_D'\, \cong"] & RK^{j+1}(B\G,B) \arrow[d, "\beta'"] 
        \\
		\Hom(K_*(B\G),K_*(A))\arrow[r,"h_*"] & \Hom(K_*(B\G),K_*(D)) \arrow[r, " "] & \Hom(K_*(B\G),K_*(B)) 
	\end{tikzcd}
\]
By \eqref{dualii}, $\mu_A^*\circ g_A=g_A'\circ \nu_A$ and $\mu_D^*\circ g_D=g_D'\circ \nu_D$.
It follows that
\[\coker\left(h_*:\Hom(K_*(C^*(\G)),K_*(A)) \to \Hom(K_*(C^*(\G)),K_*(D))\right)=\Ext(K_*(C^*(\G),K_*(B)))\]
identifies with
\[\coker\left(h_*:\Hom(K_*(B\G),K_*(A)) \to \Hom(K_*(B\G),K_*(D))\right)=\Ext(K_*(B\G),K_*(B)))\]	
via a map which is exactly $\mu_{ext}^*$.
\end{proof}
 	{Following Kasparov \cite{Kasparov-conspectus}, we use the notation $LK^*(B\G;B)=\varprojlim K^*(Y_i;B).$}
        \begin{proposition}\label{prop=amena}
        Let $\G$ be a countable torsion-free group which admits a $\gamma$-element equal to $1$ and such that $C^*(\G)$ satisfies the UCT.
		For example $\G$ is a countable torsion-free amenable group. Then there is an exact sequence 
       \[\begin{tikzcd}
  	0 \ar[r]& \mathrm{Pext}(K_*(C^*(\G)), K_*(B)) \ar[r, "f"] & KK^*(C^*(\G),B)
  	\ar[r, "\beta\circ \nu"] & LK^*(B\G;B) \ar[r]& 0
  \end{tikzcd}. \] 
	\end{proposition}
    \begin{proof} Under the assumption $\gamma=1$, it is known that both maps $\nu$ and $\mu$ are isomorphisms \cite{Kas:inv}, \cite{Tu:BC}. Since the diagram \eqref{diag-1} is commutative and the map $\alpha$ is injective, an easy diagram chase (or the nine lemma) shows that $f$ maps isomorphically the kernel of $\alpha_{ext}\circ \mu_{ext}^*$ onto the kernel of $\beta \circ \nu$. Thus we need to show that the kernel of $\alpha_{ext}\circ \mu_{ext}^*$ is the subgroup $\text{Pext}(K_*(C^*(\G)), K_*(B))$ of $\text{Ext}(K_*(C^*(\G)), K_*(B)).$
      If $G$ is the injective limit of a system $(G_i)$ of finitely generated groups, then by \cite[Prop. 5.6]{Schochet-primer}: 
	\[
	0 \to \text{Pext}(G, H) \to \text{Ext}(G, H) \xrightarrow{} \varprojlim \text{Ext}(G_i, H) \to 0\]
	Thus the kernel of the map $\alpha_{ext}$ in diagram \eqref{diag-1} is $\text{Pext}(K_*(B\G), K_*(D))$.
	By commutativity of the diagram \eqref{diag-1} and naturality of the $\text{Pext}$-functor, $\text{Pext}(K_*(B\G), K_*(D))$ coincide with the image $\text{Pext}(K_*(C^*(\G)), K_*(D))$ of under $\mu_{ext}^*$.  If $\G$ is amenable and torsion free, then $\gamma=1$ and $C^*(\G)$ satisfies the UCT by \cite{HigKas:BC}, \cite{Tu:BC}. This concludes the proof.
    \end{proof}
    
 	As explained in the prof of Lemma 3.4 from \cite{Kasparov-Skandalis-kk}, there a Milnor lim$^1$-exact sequence 
 	\[\begin{tikzcd}
 	0\ar[r] &	\varprojlim^1 K^{j-1}(Y_i;B)\ar[r] & RK^j(B\G;B)
 		\ar[r] &   \varprojlim K^j(Y_i;B) \ar[r] &0
 	\end{tikzcd}.\]
     \begin{notation}\label{notation-vb}
    Fix a classifying map $e$ for $L_\G$ as in equation~\eqref{class-map}. If $Y\subset B\G$ is compact, $S\subset \G$ finite and $\delta>0,$ we say that $(Y,S,\delta)$ is a K-triple for $\G$
    if for any unital $(S,\delta)$-representation  $\rho:\G \to U(k),$ 
    \begin{equation}\label{class-map-s}
    h=\sum_{a,b\in I}  \rho(s_{ab}) \otimes \chi_a\chi_b \otimes e_{ab},
\end{equation}
satisfies $\|h(x)^2-h(x)\|<2/9$ for all  $x\in Y$.
We denote by $e^Y_\rho\in C(Y) \otimes M_\infty(\C)$ the projection obtained from $h|_Y$ via functional calculus. The corresponding vector bundle is denoted by $E^Y_\rho.$ It easy to see that for each compact $Y,$ there are $S$ and $\delta$ such $(Y,S,\delta)$ is a K-triple.
   \end{notation}
  Recall that we fixed a decomposition $B\G=\bigcup_i Y_i$, see \eqref{union}.
   Consider a sequence of pairs $(S_n,\delta_n)$ such that $\G=\bigcup_n S_n,$ and $\delta_n \searrow 0.$  
   Consider an asymptotic homomorphism consisting of a sequence of
   unital maps $\rho_n:\G \to U(k_n)$ which are $(S_n,\delta_n)$-representations. This induces 
   a unital $*$-homomorphism 
   \begin{equation}\label{rho-dot}
   \dot\rho:C^*(\G)\to B=\prod_n M_{k_n}(\C)/\bigoplus_n M_{k_n}(\C).
   \end{equation}
 
  Since the subspaces $Y_i$ are compact, after passing to a subsequence we may arrange that
  for each $n$,  $(Y_n,S_n,\delta_n)$ is a K-triple for $\G.$
  Thus if $e_n:=e|_{Y_n},$ then 
  $h_{n}=({\rho_n}\otimes\id{} )(e_n)$ satisfies
  $\|h_{n}^2-h_{n}\|<2/9.$
  We denote the vector bundle on $Y_n$ associated to $p_n$ and implicitly to $\rho_n$ by $E_{\rho_n}$.
  For $i\leq n$, observe that
   $h_{n,i}=(\id{}\otimes {\rho_n})(e_i)$ satisfies
   $\|h_{n,i}^2-h_{n,i}\|<2/9,\quad \forall\, i\leq n,$ since $h_{n,i}=h_n|_{Y_i}$.
  Let $p_{n}\in M_{k_n}\otimes C(Y_{n}) \otimes \K$ and $p_{n,i}\in M_{k_n}\otimes C(Y_{i}) \otimes \K$ be the projections $p_{n}=\chi(h_{n})$ and $p_{n,i}=\chi(h_{n,i})$ obtained by functional calculus, where $\chi$ is the characteristic function of $(0.5,1.5)$.
  Then $\|h_{n}-p_{n}\|<1/3$ and $\|h_{n,i}-p_{n,i}\|<1/3.$
  Let us note that 
  \begin{equation}\label{eq-restriction}
   p_n|_{Y_i}=p_{n,i}\quad \text{for all} \,\, i\leq n,    
  \end{equation}
   by functoriality of the functional calculus.
   If we have another asymptotic homomorphism $\{\rho'_n:\G \to U(k_n)\}_n$ with $\rho'_n$ a $(S_n,\delta_n)$-representation with the same properties as above, then,
   we construct $h'_{n,i}$ and $p'_{n,i}$ similarly. In particular if follows from ~\eqref{eq-restriction} that if
   \begin{equation}
   [p_n]=[p_n']\in K_0(C(Y_n))
   \end{equation}
   then 
   \begin{equation}\label{ni}
   	[p_{n,i}]=[p_{n,i}']\in K_0(C(Y_i)), \quad \text{for all} \,\, i\leq n.
   \end{equation}

   \begin{lemma}\label{lemma:ML}
	If $A$ is a separable and nuclear $C^*$-algebras satisfying the UCT, then the following natural map is injective.
	\[\theta: K_*\left(A\otimes \frac{\prod_n M_{k_n}(\C)}{\bigoplus_n M_{k_n}(\C)}\right)\to \frac{\prod_n K_*(A)}{\bigoplus_n K_*(A)}\]
\end{lemma}
\begin{proof} Let $B=\prod_n M_{k_n}(\C)/\bigoplus_n M_{k_n}(\C).$ Then $K_1(B)=0$ and $K_0(B)$ is torsion free since
	\[K_0\left(\prod_n M_{k_n}(\C)\right)\cong \{(x_n)_n\in \prod_n \Z: (x_n/k_n)_n\}\,\, \text{is a bounded sequence}\}.\]
	Using the K\"unneth formula and the property that $K_0(B)$ is a pure subgroup of ${\prod_n \Z}/{\bigoplus_n \Z}$, for $j=0,1$, one shows  that the  map
	\[K_j(A\otimes B)\cong K_j(A)\otimes K_0 (B)\to K_j(A)\otimes \left(\frac{\prod_n \Z}{\bigoplus_n \Z}\right)\to \frac{\prod_n K_j(A)}{\bigoplus_n K_j(A)}\]
	is injective. 	For the injectivity of the map on the right, we use the property that if $G$ is an abelian groups, then the natural map 
    \begin{equation*}\label{given-map}
       \imath: G\otimes \prod_n \Z \to \prod_n G
    \end{equation*}
     is injective.
    Indeed,  any $x\in G\otimes \prod_n \Z$ is in the image of $G_x\tensor\prod_n\Z,$ for some finitely generated subgroup $G_x$ of $G,$ and we have a commutative diagram
    $$\begin{tikzcd}
G_x\tensor\prod_n\Z\arrow{r}{}\arrow{d}{}&\prod_nG_x\arrow{d}{}\\
G\tensor\prod_n\Z\arrow{r}{\imath}&\prod_nG
\end{tikzcd}$$
It is routine to verify that the top map is injective, as $G_x$ is a finite sum of cyclic groups. 
This is also implied by the easy direction of \cite[Thm. 8.14]{Fuchs-ab}. The map on the right is clearly injective. It follows that the composition of the map on the left and $\imath$ is injective as well. Thus if $x\neq 0,$ then $\imath(x)\neq 0$.
     \end{proof}

  \begin{lemma}\label{lemma:22}
  	Consider two asymptotic homomorphisms $\{\rho_n,\rho_n':\G \to M_{k_n}(\C)\}_n$ which have the same approximate multiplicativity properties as above. Thus for each $n$,  
    $\rho_n,\rho_n'$ are  $(S_n,\delta_n)$-representations  and $(Y_n,S_n,\delta_n)$ is a K-triple for $\G.$
    Suppose that $[p_n]=[p_n']\in K_0(C(Y_n))$ for all $n\in \N$. Then with $\dot{\rho}, \dot{\rho}' :C^*(\G)\to B$ defined as in \eqref{rho-dot} we have 
  	\[[(\dot{\rho}\otimes\id{} )(e_i)]=[(\dot{\rho}'\otimes\id{} )(e_i)]\in K_0(B \otimes C(Y_i)),  \quad \text{for all} \,\, i\in \N.\]
  \end{lemma}
   \begin{proof} 
   The map $\theta$ in the diagram below is injective by Lemma~\ref{lemma:ML}.
   \begin{equation}
  \begin{tikzcd}
  	K_0(C^*(\G)\otimes C(Y_i)) \ar[r, "\dot{\rho}"] & K_0(B\otimes C(Y_i))
  	\ar[r, "\theta"] & \frac{\prod_n K_*(C(Y_i))}{\bigoplus_n K_*(C(Y_i))}
  \end{tikzcd} 
   \end{equation}
   Therefore it suffices to show that $\theta[(\dot{\rho}\otimes\id{} )(e_i)]=\theta[(\dot{\rho}'\otimes\id{} )(e_i)].$
   By construction, the projection $(\dot{\rho}\otimes\id{} )(e_i)$ lifts to the element $({\rho}\otimes\id{} )(e_i)=(h_{n,i})_{n\geq i} \in \left(\prod_{n\geq i} M_{k_n}(\C) \right)\otimes C(Y_i)\otimes \K$ and therefore to the projection $(p_{n,i})_{n\geq i}$. 
   Thus $\theta[(\dot{\rho}\otimes\id{} )(e_i)]$ is represented by the sequence $	([p_{n,i}])_{n\geq i}.$
   Similarly,  $\theta[(\dot{\rho}'\otimes\id{} )(e_i)]$ is represented by the sequence $	([p_{n,i}'])_{n\geq i}.$
   By assumption, $[p_n]=[p_n']\in K_0(C(Y_n))$ for all $n$ and hence $	[p_{n,i}]=[p_{n,i}']\in K_0(C(Y_i)), \,\, \text{for all} \,\, n\geq i,$ by \eqref{ni}.
   	 \end{proof}

    Let  $F\subset \G$ be a finite set  and let $\ep>0.$ For two maps $\varphi,\psi:\Gamma\to U(B),$ we write
    \[\varphi\underset{F, \varepsilon}{\approx}\psi\]
    if there is a unitary $u\in U(B)$ such that
    $\|u\varphi(s)u^*-\psi(s)\|<\ep,\,\, \forall s \in F.$ We shall use a similar notation to indicate approximate unitary equivalence for maps between $C^*$-algebras.

\begin{theorem}\label{uniqlo} Let $\G$ be a torsion-free residually finite countable amenable group. For any finite set $F\subset \G$ and any $\ep>0,$ there exist a compact subspace $Y\subset B\G$, a finite set $S\subset \G$ and $\delta>0$ with $(Y,S,\delta)$ a K-triple,
such that 
  for any two $(S,\delta)$-representations
$\rho,\rho':\G \to U(k)$ with $[E^Y_\rho]=[E^Y_{\rho'}]$ in $K^0(Y),$ there is a representation $\pi:\G \to U(m)$ and a unitary $u\in U(k+m)$ such that 
\begin{equation}\label{eq:ssuu}
\|u(\rho(s)\oplus \pi(s))u^*-\rho'(s)\oplus \pi(s)\|<\ep,\quad \forall s\in F.
	\end{equation}
    \end{theorem}
    \begin{proof}
     We prove this by contradiction. Suppose that there are $F$ and $\ep$ for which no $S,\delta$ and $Y$ satisfy the conclusion of the Theorem. Write
     $B\G=\bigcup_i Y_i$, as in \eqref{union}.
      Choose a sequence of pairs $(S_n,\delta_n)$ such that $\G=\bigcup_n S_n,$ and $\delta_n \searrow 0$ so that each $(Y_n,S_n,\delta_n)$ a K-triple. Then the vector bundle $E_{\rho_n}$ as in Notation~\ref{notation-vb} is well defined for any $(S_n,\delta_n)$-representation. By assumption,  there are two sequences of $(S_n,\delta_n)$-representation  $\rho_n,\rho_n':\G\to U(k_n),$ such that $[E_{\rho_n}^{Y_n}]=[E_{\rho_n'}^{Y_n}]$ in   $K^0(Y_n)$ and yet 
\begin{equation}\label{eq:ssuq11}
\rho_n\oplus \pi_n\underset{F, \varepsilon}{\not\approx}\rho'_n\oplus \pi_n
	\end{equation}
for any finite dimensional representation $\pi_n$.

        Let $B=\prod_n M_{k_n}(\C)/\bigoplus_n M_{k_n}(\C)$  and let $\dot\rho,\dot\rho':\G \to U(B)$ be the  $*$-homomorphisms induced by the sequences $(\rho_n)$ and $(\rho'_n)$.
   By Lemma~\ref{lemma:22},
      the condition $[E_{\rho_n}^{Y_n}]=[E_{\rho_n'}^{Y_n}],$ for all $n\geq 1,$ implies that
		$[\dot\rho]-[\dot\rho']$ 
        belongs to the kernel of the map $KK^*(C^*(\G),K_*(B))
  	\to\varprojlim K^*(Y_n;B)$ and hence 
    $[\dot{\rho}]-[\dot{\rho}']\in \text{Pext}(K_*(C^*(\G)), K_*(D))$ by Proposition~\ref{prop=amena}. 
    Since $C^*(\G)$ satisfies the UCT,
        $ \text{Pext}(K_*(C^*(\G)), K_*(D))=\overline{\{0\}}$ by \cite[Thm. 4.1]{Dad:kk-top}.
        Since $\G$ is residually finite and amenable, $C^*(\G)$ is residually finite dimensional.
        The stable uniqueness result from \cite[Cor. 3.8]{Dad:kk-top} implies 
        that for any finite set $F\subset \G$ and any $\ep>0,$ there is a finite dimensional representation $\pi:C^*(\G)\to M_m(\C)\subset M_m(\C 1_B)$ such that
        \[\dot\rho\oplus \pi \underset{F, \varepsilon}{\approx}\dot\rho'\oplus \pi.\]
        This property contradicts \eqref{eq:ssuq11} for all sufficiently large $n$.
    \end{proof}

  \begin{theorem}\label{exi}  Let $\G$ be a torsion-free residually finite countable amenable group.
        For any $(z_n)_n\in \varprojlim \widetilde{K}^0(Y_n)=L\tilde{K}^0(B\G) $, there exists an asymptotic representation $\{\rho_n:\G \to U(k_n)\}_n$ such that $[E_{\rho_n}^{Y_n}]-[k_n]=z_n$ for all $n\geq 1$. Moreover,
       by Theorem~\ref{uniqlo}, $\{\rho_n\}_n$ is unique up to stable approximate unitary equivalence. Thus, if $\{\rho'_n:\G \to U(k_n)\}_n$ is another lifting of $(z_n)_n$, then there exist a sequence of representations $\{\pi_n:\G \to U(\ell_n)\}_n$ and unitaries $u_n\in U(k_n+\ell_n)$ such that 
       \[\lim_{n\to \infty} \|\rho_n(s)\oplus \pi_n(s)-u_n\left( \rho'_n(s)\oplus \pi_n(s)\right)u_n^*\|=0,\quad \forall s\in \G.\]
    \end{theorem}
    
     \begin{proof}
      The existence part is proved implicitly in \cite{AA}. For the sake of completeness, we review the argument here. Since $\G$ is amenable and torsion free, the map $KK(C^*(\G),\C)\to \varprojlim K^0(Y_n)$ is surjective by Proposition~\ref{prop=amena}. Therefore there is $\alpha\in KK(C^*(\G),\C)$ such that 
      $\nu_{n}(\alpha)=z_n$ for all $n\geq 1$. Represent $\alpha$ as the class of a Cuntz pair $\alpha=[\varphi,\psi]$, $\varphi,\psi:C^*(\G)\to B(H)$, $\varphi(a)-\psi(a)\in \K(H)$, $a\in C^*(\G)$. By \cite{Skandalis:K-nuclear} we can choose $\psi$ to be a fixed faithful essential representation of $C^*(\G)$. Since $\G$ is amenable and residually finite, there is such a $\psi$ which is direct sum of finite dimensional representations. Let $(p_n)$ be an increasing approximate unit of $\K(H)$ consisting of projections which commutes with $\psi$. Then $\varphi_n(a)=p_n\varphi(a)p_n$ is an asymptotic homomorphism and each $\psi_n(a)=p_n\psi(a)p_n$ is a finite dimensional representation. 
      After passing to subsequences, Proposition 2.5 of \cite{AA} implies that there is $n_0\in \N$ such that for all $n\geq n_0:$
      \[z_n=\nu_n(\alpha)=[E_{\varphi_n}^{Y_n}]-[E_{\psi_n}^{Y_n}].\]
      Since all rational Chern classes of a flat complex bundle do vanish, \cite{Dupont},
       for each $n$ there is an integer $r_n\geq 1$ such that the bundle $(E_{\psi_n}^{Y_n})^{\oplus r_n}$ is trivial. We conclude the proof by choosing $\rho_n=\varphi_n\oplus \psi_n^{\oplus (r_n-1)}.$
   \end{proof}
   \begin{remark}
As noted in the introduction, Theorems~\ref{uniqlo} and~\ref{exi} extend to torsion-free amenable groups that are not necessarily residually finite, at the cost of replacing the representations~$\pi_n$ with asymptotic representations. The proofs are very similar. For the existence part, one chooses $\psi = \lambda_\Gamma$ to be the left regular representation of~$\G$.  
By the Tikuisis--White--Winter theorem~\cite{TWW:quasidiagonality}, $\lambda_\Gamma$ is a quasidiagonal representation, and hence we can choose $(p_n)$ to be an increasing approximate unit of~$\K(H)$ consisting of projections that commute asymptotically with both~$\varphi$ and~$\psi$. For the uniqueness part, one uses the stable uniqueness theorem from~\cite{DadEil:AKK}, with $\lambda_\G$ playing the role of the absorbing representation. We do not state these more general results explicitly here, as we find them less elegant.
\end{remark}
   \begin{corollary}
       \label{thm:existence} Let $\G$ be a  residually finite countable amenable group such that $B\G$ is compact  and let $z\in \widetilde{K}^0(B\G)$. For any  finite set $S\subset \G$ and any $\delta>0$ there is an $(S,\delta)$-representation $\rho:\G \to U(n)$  such that $[E_{\rho}]-[n]=z.$ 
   In particular for any $x\in \widetilde{H}^{\text{even}}(B\G,\Q)$ there is an $(S,\delta)$-representation $\rho:\G \to U(n)$ such that $\widetilde{\ch}(E_\rho)=qx$ for some $q\in \Q$.
 \end{corollary}
 
\section{Approximation by projective representations}
In this section we prove Theorem~\ref{thm:proja}.
		\begin{proposition}\label{prop-proj}
		Let $\G$ be a virtually polycyclic group and let $Y\xhookrightarrow{i} B\G$ be a compact subspace. 
For any $x\in H^2(B\G,\Z)$ there is an asymptotic representation consisting of  projective representations $(\psi_n:\G \to U(m_n))_n$ such that 
 $c_1(E_{\psi_n}^Y)=\frac{m_n}{n}i^*(x)\in H^2(Y,\Q),$  and hence $\ch(E_{\psi_n})=m_ne^{\frac1{n}i^*(x)}\in H^{\text{even}}(Y,\Q),$ for all sufficiently large $n\in \N$.
\end{proposition}
\begin{proof} We are using a method from \cite{arXiv:2204.10354}.
	Let $x$ be represented by a $2$-cocycle $\sigma\in Z^2(\G,\Z)$ and construct the corresponding central extension
	\begin{equation*}
		\begin{tikzcd}
			0\arrow{r}{}&\Z\arrow{r}{\iota}&\G_\sigma\arrow{r}{}&\G\arrow{r}&1,
		\end{tikzcd}
	\end{equation*}
	so that for a set-theoretic splitting $\gamma:\G \to \G_\sigma,$
	$\gamma(a)\gamma(b)\gamma({ab})^{-1}=\iota(\sigma(a,b))$, for all $a,b\in \G$.
	For each $n\geq 1$, 
 the cyclic subgroup $\langle\iota(1)^n\rangle$ generated by  the central element $\iota(1)^n$ is a normal subgroup. Clearly, $\G_\sigma/\langle\iota(1)^n\rangle$ is virtually polycyclic and thus residually finite by \cite{Hirsch}. It follows that there is a finite quotient $G_n$ of $\G_\sigma$ where the image of $\iota(1)$ generates a central subgroup isomorphic to $\Z/n$. Let $\lambda_n$ be the representation of $G_n$ induced by the character of $\Z/n$ that maps the generator to $e^{2\pi i/n}$. The restriction of $\lambda_n$ to $\Z/n$ is a multiple of that character, with multiplicity $m_n=[G_n:\Z/n]$.
	The composition   
	\begin{equation*}
		\begin{tikzcd}
	\psi_n:\G\arrow{r}{\gamma} &	\G_\sigma\arrow{r}&G_n\arrow{r}{\lambda_{n}}&U(m_n),
		\end{tikzcd}
	\end{equation*}
	is a projective representation of $\G$ that satisfies 	$\psi_n(a)\psi_n(b)\psi_n({ab})^{-1}=e^{\frac{2\pi i}{n}\iota(\sigma(a,b))}1_{m_n}.$
	It follows by Theorem~\ref{thm:c11}  that $c_1(E_{\psi_n})=\frac{m_n}{n}\,x\in H^2(B\G,\Q).$
Note that $\|\sigma\|_F=\sup_{a,b\in F}|\sigma(a,b)|$ then for sufficiently large $n$:
\[\sup_{a,b\in F}\|\psi_n(a)\psi_n(b)-\psi_n({ab})\|\leq \frac{4\pi}{n}\|\sigma\|_F.\qedhere\]
\end{proof}

When we will use this proposition in the sequel, we will write $m_n^x$ for $m_n,$ to indicate that the sequence is associated to $x$. Now that a sequence $(m_n^x)$ was found, it is clear that the conclusion of the statement remains true if we replace each $(m_n^x)$ by a sequence $(m_n)$ where each $m_n$ is a multiple of $m_n^x$.

\begin{lemma}\label{algebra}
		  Let $Y$ be a finite CW complex. 
          \begin{enumerate}
              \item $\widetilde{H}^{\text{even}}(Y;\Q)$ is spanned by nonnegative linear combinations of elements of the form $e^x-1$ with $x\in H^2(Y;\Q)$ if and only if for any $y\in \widetilde{H}^{\text{even}}(Y;\Q)$ there are finitely many elements $x_i\in H^2(Y,\Z)$ and natural numbers $k_i$ such that $ry=\sum_i k_i(e^{x_i}-1)$ for some integer $r\geq 1$.
          
 \item $\widetilde{H}^{\text{even}}(Y;\Q)$ is spanned by linear combinations of elements of the form $e^x-1$ with $x\in H^2(Y;\Q)$ if and only if for any $y\in \widetilde{H}^{\text{even}}(Y;\Q)$ there are finitely many elements $x_i\in H^2(Y,\Z)$ and integers $k_i$ such that $ry=\sum_i k_i(e^{x_i}-1)$ for some integer $r$.
 
 \item Observe that if $y\in {H}^{2k}(Y;\Z),$ $k\geq 1$ and $ry=\sum_i k_i(e^{x_i}-1)$ as in (1) above, then 
$rm^ky=\sum_i k_i(e^{mx_i}-1)$ for any $m\in \Q$.
          \end{enumerate}
		\end{lemma}

\begin{proof}
{\em(3)} For $x\in\widetilde{H}^{\text{even}}(Y;\Q)$ we will let $[x]_{2k}$ denote the component of $x$ that is in $H^{2k}(Y;\Q)$. Then for $j\ne k$, 
$$rm^k[y]_{2j}=0=rm^j[y]_{2j}=\sum_i k_i[e^{m x_i}]_{2j},$$
and one can easily see that $rm^k[y]_{2k}=\sum_ik_i[e^{mx_i}]_{2k}$. Thus, the desired equality holds for all $j$. 

{\em(1)} We will only show the ``only if'' direction since the ``if'' direction is obvious. It is sufficient to show this for $y\in H^{2k}(Y;\Q)$ for some $k>0$. By assumption $y=\sum_i\kappa_i(e^{\chi_i}-1)$ with $\kappa_i,\chi_i\in\Q^+$. Let $d$ be the least common multiple of the denominators of the $\chi_i$ terms, and observe that by {\em(3)}, $d^ky=\sum_i\kappa_i( e^{\chi_id}-1)$. Multiplying both sides by the least common multiple of the denominators of the $\kappa_i$ terms, we get the desired result. 

The proof of {\em(2)} is identical to that of {\em(1)}.
\end{proof}

It is known that every polycyclic group is a virtually poly-$\Z$ group.
	An exact sequence of discrete groups
			$ 1 \longrightarrow N \longrightarrow G \longrightarrow Q \longrightarrow 1$ gives a locally trivial fiber bundle of classifying spaces: $BN \longrightarrow BG \longrightarrow BQ.$ From this we see that if $\G$ is a   poly-$\Z$ group,  then $B\G$ can be realized as a finite CW-complex.

\begin{notation}
If $\rho:\G\rightarrow U(n)$ is a function and $r\in\N$, we write $r\rho:=\rho^{\oplus r}$.
\end{notation}
The following implies Theorem~\ref{thm:proja}
from the introduction.
		\begin{theorem} \label{thm:projb}
			Let $\G$ be a virtually polycyclic group with $B\G$ a finite CW complex. Consider the following properties:
			\begin{enumerate}
				\item For every $\varepsilon>0$ and finite subset $F\subset\G$, there exist a finite subset $S\subset\G$ and $\delta>0$ so that for any $(S,\delta)$-representation $\rho$, there are an integer $r>0$,  representations  $\pi_1$ and $\pi_2$, a unitary $u$, and  a projective representation $\psi$, so that
				$$||u(r \rho(s)\oplus\pi_1(s))u^*-\psi(s)\oplus\pi_2||<\varepsilon,\quad\forall s\in F.$$
				\item For every $\varepsilon>0$ and finite subset $F\subset\G$, there exist a finite subset $S\subset\G$ and a $\delta>0$ so that for any $(S,\delta)$-representation $\rho$, there are an integer $r>0$, a unitary $u$,  a representation $\pi,$ and a finite family of projective representations $(\psi_i)$,  so that
				$$\left\|u(r \rho(s)\oplus\pi(s))u^*-\bigoplus_i\psi_i(s)\right\|<\varepsilon\quad\forall s\in F.$$
				\item For every $\varepsilon>0$ and finite subset $F\subset\G$, there exist a finite subset $S\subset\G$ and a $\delta>0$ so that for any $(S,\delta)$-representation $\rho$, there are an integer $r>0$, a unitary $u$, and finite families of projective representations $(\varphi_i)$ and $(\psi_k)$ so that
				$$\left\|u\left(r\rho(s)\oplus\bigoplus_i\varphi_i(s)\right)u^*-\bigoplus_k\psi_k(s)\right\|<\varepsilon\quad\forall s\in F.$$
			\end{enumerate}
			Then condition (1) is true if and only if $\widetilde{H}^{\text{even}}(\G;\Q)=H^2(\G;\Q)$. Condition (2) is true if and only if $\widetilde{H}^{\text{even}}(\G;\Q)$ is spanned by nonnegative linear combinations of elements of the form $e^x-1$ with $x\in H^2(\G;\Q)$. Condition (3) is true if and only if $\widetilde{H}^{\text{even}}(\G;\Q)$ is spanned by linear combinations of elements of the form $e^x$ with $x\in H^2(\G;\Q)$.
		\end{theorem}
		\begin{proof}  	
			Suppose that condition (1) holds. Let $z\in H^{2k}(B\G,\Q)$ for some $k>1$. By Corollary~\ref{thm:existence}, see also \cite{AA}, there is  an asymptotic representation $(\rho_n)_n$ such that $\widetilde{\ch}(E_{\rho_n})=q_nz,$ for some $q_n\in \Q$, for  all $n\in \N$. 
		By assumption, there are positive integers 
	 $(r_n)$, representations $(\pi_n)^{(1)}$ and $(\pi_n^1)$ and projective representations $(\psi_n^2)$ so that
			$$||r_n \rho_n(s)\oplus\pi_n^1(s)-\psi_n(s)\oplus\pi_n^2||\longrightarrow0 ,\quad\forall s\in \G.$$
	
	Note that $(\psi_n)$ is necessarily an asymptotic representation. It follows that for all sufficiently large $n$, $r_n[E_{\rho_n}]+[E_{\pi_n^1}]=[E_{\psi_n}]+[E_{\pi_n^2}].$ Since $E_{\pi_n^i}$ is a flat vector bundles, we obtain that $r_n\widetilde{\ch}(E_{\rho_n})=\widetilde{\ch}(E_{\psi_n})$ and in particular $c_1(E_{\psi_n})=c_1(E_{\rho_n})=0$.  Since $(\psi_n)$ are projective representations with $c_1(E_{\psi_n})=0$, it follows from Theorem~\ref{thm:c11} that $\widetilde{\ch}(E_{\psi_n})=0$. Therefore $r_n\widetilde{\ch}(E_{\rho_n})=r_nqz=0$ and hence $z=0$.
	
	Conversely, assume now that $\widetilde{H}^{\text{even}}(\G;\Q)=H^2(\G;\Q)$.
	Then for $F\subset\G$ finite and $\varepsilon>0$,   pick the $S\subset\G$ finite and $\delta>0$ according to Corollary~\ref{thm:ex&unique}. Suppose that $\rho$ is an $(S,\delta)$-representation and set  $x:=c_1(E_\rho)_{\Z}\in H^2(B\G,\Z).$ By Proposition~\ref{prop-proj}, there is an asymptotic representation consisting of  projective representations $(\psi_n:\G \to U(m_n))_n$ such that $c_1(E_{\psi_n})=\frac{m_n}nx\in H^2(B\G,\Q),$ for all sufficiently large $n\in \N$. It follows that $\widetilde{\ch}(m_n[E_\rho]-n[E_{\psi_n}])=0.$ Since the Chern character is a rational isomorphism,  it follows that $[E_{p_n\rho}]=p_n[E_\rho]=q_n[E_{\psi_n}]+k_n[1]=[E_{q_n{\psi_n}}]+k_n[1]$ for some integers $p_n,q_n>0$ and $k_n\in\Z$.
We conclude the proof by applying Corollary~\ref{thm:ex&unique} for the approximate representations $p_n\rho$ and  $q_n{\psi_n},$ for a suitable large $n$ which assures that $\psi_n$ is a $(S,\delta)$-representation.

Suppose that condition (2) holds. 
Then for any given $z\in\widetilde{H}^{\text{even}}(\G;\Q)$, there is some asymptotic representation $\rho_n$ and a rational numbers $q_n$ with $z=q_n\widetilde{\ch}(E_{\rho_n})$. By condition (2), there are sequences of natural numbers $N_n$ and $r_n$, projective representations $\{\psi_{n}^i\}_{i=1}^{N_n}$ and representations $\pi_n$,  so that
$$\left\|u_n(r_n\rho_n(s)\oplus\pi_n(s))u_n^*-\bigoplus_{i=1}^{N_n}\psi_n^i(s)\right\|\rightarrow 0,\quad\forall s\in \G.$$  Thus, for large enough $n$ it follows that 
$$r_nz=\sum_{i=1}^{N_n}\widetilde{\ch}(E_{\psi_n^i})$$
and so the desired result follows from Theorem~\ref{thm:c11}.

Conversely, let $\varepsilon>0$ and let $F\subset\G$ be finite. Pick the $\delta>0$ and $S\subset\G$ finite according to Corollary~\ref{thm:ex&unique} and suppose that $\rho$ is an $(S,\delta)$-representation. Let $z=\widetilde{\ch}(E_\rho)$.
Let $z=\widetilde{\ch}_k(E_\rho)\in H^{2k}(B\G,\Z)$ for some $k\geq 1$.
By Lemma~\ref{algebra}, there are finitely many elements $x_i\in H^2(Y,\Z)$ and natural numbers $k_i$ such that $qz=\sum_i k_i(e^{x_i}-1)$ for some integer $q\geq 1$. For each $n\geq 1$, let $m_n=\prod_i m_n^{x_i}$ where $m_n^{x_i}$ are given by Proposition~\ref{prop-proj}. Then ${m_n}qz=n^k\sum_i m_nk_i(e^{\frac1nx_i}-1).$ By applying Proposition~\ref{prop-proj} and selecting a large enough $n$, we find a finite family $(\psi_i)$ of projective representations which are $(S,\delta)$-multiplicative and such that $\ch(E_{\psi_i})=m_n e^{\frac1nx_i}$. Repeating this process for each $k>0$ with $\ch_k(E_\rho)\ne0$ and replacing $k_i$ with $n^k\cdot k_i$ we get, $m_n\widetilde{\ch}(E_\rho)=\sum_im_nk_i(e^{\frac1n x_i}-1)$. Thus, if we set $p=m_nq/n^k$, then $[E_{p\rho}]-[\bigoplus_i E_{\psi_i}]$ is a torsion element plus a trivial bundle of $K^0(B\G)$. Thus there is a multiple $r$ of $p$ such that if we set $r'=r/p$, then
$[E_{r\rho}]=[\bigoplus_i E_{r'\psi_i}]$ in $K^0(B\G)$.
We apply Corollary~\ref{thm:ex&unique} to conclude the desired result.

Suppose that condition (3) holds. Following the same reasoning as above, we get that for any $z\in\widetilde{H}^{\text{even}}(\G;\Q)$, we can write
$$rz=\sum_{i=1}^N\widetilde\ch(E_{\psi_i})-\sum_{k=1}^M\widetilde\ch(E_{\varphi_j})$$
for projective representations $\psi_i$ and $\varphi_j$ and an integer $r\geq 1$. Then the desired result follows from Theorem~\ref{thm:c11}.

The proof of the converse is similar to proof of part (2) above. Let $\varepsilon>0$ and let $F\subset\G$ be finite. Pick the $\delta>0$ and $S\subset\G$ finite according to Corollary~\ref{thm:ex&unique} and suppose that $\rho$ is an $(S,\delta)$-representation. 
We may assume that $z=\widetilde{\ch}(E_\rho)\in H^{2k}(B\G,\Z)$ for some $k\geq 1$.
By Lemma~\ref{algebra}, there are finitely many elements $x_i,y_j\in H^2(Y,\Z)$ and natural numbers $k_i,\ell_j$ such that $qz=\sum_i k_i(e^{x_i}-1)-\sum_j \ell_j(e^{y_j}-1)$ for some integer $q\geq 1$. Let $m_n=\prod_{i,j} m_n^{x_i} m_n^{y_j}$ where $m_n^{x_i}, m_n^{y_j}$ are given by Proposition~\ref{prop-proj}. Then $\frac{m_n}{n^k}qz+\sum_j \ell_jm_n(e^{\frac1ny_j}-1)=\sum_i k_i(m_ne^{\frac1nx_i}-1).$
By applying Proposition~\ref{prop-proj} and selecting a large enough $n$, we find  finite families $(\psi_i)$ and $(\varphi_j)$ of projective representations which are $(S,\delta)$-multiplicative and such that $\ch(E_{\psi_i})= m_ne^{\frac1nx_i}$ and $\ch(E_{\varphi_j})=m_ne^{\frac1ny_j}$. Reasoning as above
the K-theory classes $[E_{p\rho} \oplus \bigoplus_j E_{\varphi_j}]$ and $[\bigoplus_i E_{\psi_i}]$ differ by a torsion element plus a trivial bundle. We conclude the proof by applying Corollary~\ref{thm:ex&unique} to a common multiple of these almost representations, just as in (2) above.
\end{proof}


\section{Examples of almost representations observing higher invariants}
In this section we assume that $B\G$ is compact.
Let us denote by $\mathrm{Rep}^k_{(S,\ep)}(\G)$ the set of unital maps
 such that $\|\rho(st)-\rho(s)\rho(t)\|<\varepsilon$ for all $s,t \in S$
 with the property that the vector bundle $E_\rho=E_\rho^{B\G}$ associated to $\rho$ as in Definition~\ref{def:pushf} is well-defined. 
 Let $\mathrm{Rep}_{(S,\ep)}(\G)$ be the disjoint union $\bigsqcup_{k}\mathrm{Rep}^k_{(S,\ep)}(\G).$ Henceforth, we will consider only groups with compact classifying space, and we introduce the following more concise notation in that case.
\begin{definition}\label{def:chern-pushf}
For $\rho \in \mathrm{Rep}_{(S,\ep)}(\G)$ define its Chern character by
\begin{equation}
\ch(\rho)=\ch(E_\rho)\in H^{\text{even}}(B\G,\Q)\cong H^{\text{even}}(\G,\Q).
\end{equation}
\end{definition}

The purpose of this section is to develop techniques to construct almost representations $\rho$ with the following property: $\ch_k(\rho)\ne0$ and $\ch_j(\rho)=0$ for $1\le j\le k$. This implies that $c_k(\rho)$ is the first non-vanishing Chern class of $\rho$.
Our focus will be on certain classes of nilpotent groups.
\begin{lemma}\label{lemma:funct-chern}
   (i) Suppose that $\rho_1,\rho_2 \in \mathrm{Rep}_{(S,\ep)}(\G)$ are such that 
    $\rho_1\oplus\rho_2 \in \mathrm{Rep}_{(S,\ep)}(\G)$ and $\rho_1\otimes \rho_2 \in \mathrm{Rep}_{(S,\ep)}(\G)$. Then $\ch(\rho_1\oplus\rho_2)=\ch(\rho_1)+ch(\rho_2)$,
    $\ch(\rho_1\otimes\rho_2)=\ch(\rho_1)\cdot \ch(\rho_2),$ and $\ch_k(\bar\rho)=(-1)^k\ch_k(\rho)$.
    (ii) If $f:\G \to \La$ is a homomorphism of groups, under the assumptions of Proposition~\ref{prop:pullback}, 
    \[f^*\left(\ch(\rho)\right)=\ch(\rho\circ f)\] in $K^0(B\G)\otimes \Q.$
\end{lemma}
\begin{proof}
    Part (1) follows from basic properties of the Chern character, while part (2) is a consequence of Proposition~\ref{prop:pullback}.
\end{proof}
The following fact is established in the proof of \cite[Thm. 1.2]{DDD}.
\begin{proposition}
Suppose that $B\G$ is a finite complex. Let $\{\rho_n:\G \to U(k_n)\}$ be a sequence of unital maps such that
 \begin{equation*}
     \label{a11}\lim_{n\to \infty} \|\rho_n(st)-\rho_n(s)\rho_n(t)\|=0, \quad \text{for all}\,\, s,t\in \G,\end{equation*} and 
     \(\ch(\rho_n)-k_n\neq 0,\) for all $n\geq 1$.
 Then there exists no sequence of homomorphisms $\{\pi_n:\G\to U(k_n)\}$ such that
\begin{equation*}\label{a12}\lim_{n\to \infty}  \|\varphi_n(s)-\pi_n(s)\|=0,\quad \text{for all}\,\, s\in \G.\end{equation*}
\end{proposition}
Let us recall that Voiculescu's example of a nontrivial asymptotic representation $\psi_n:\Z^2 \to U(n)$ is defined by $(x,y)\mapsto u_n^xv_n^y$ where $u_n$ and $v_n$ are the $n\times n$ matrices,

$$u_n=\begin{pmatrix}
0&0&\cdots&0&0&1\\
1&0&\cdots&0&0&0\\
0&1&\cdots&0&0&0\\
\vdots&\vdots&\ddots&\vdots&\vdots&\vdots\\
0&0&\cdots&1&0&0\\
0&0&\cdots&0&1&0
\end{pmatrix}\mbox{ and }
v_n=\begin{pmatrix}
e^{\frac{2\pi i}{n}} & 0 & 0 & \cdots & 0 & 0 \\
0 & e^{\frac{4\pi i}{n}} & 0 & \cdots & 0 & 0 \\
0 & 0 & e^{\frac{6\pi i}{n}} & \cdots & 0 & 0 \\
\vdots & \vdots & \vdots & \ddots & \vdots & \vdots \\
0 & 0 & 0 & \cdots & e^{\frac{2\pi i (n-1)}{n}} & 0 \\
0 & 0 & 0 & \cdots & 0 & 1
\end{pmatrix}.
$$

\begin{proposition}\label{formula1}
Let $x\in \widetilde{H}^{\text{even}}(\Z^{d};\Z)$. Then there is an asymptotic representation  $\rho_n:\Z^{d}\rightarrow U(m_n)$ so that $\widetilde{\ch}(\rho_n)=x$. Furthermore, $\rho_n$ can be built from tensor products and direct sums of Voiculescu's example. 
 In particular, it follows that every asymptotic representation of $\Z^d$ is stably equivalent to one built from tensor products and direct sums of Voiculescu's example.
\end{proposition}
\begin{proof}
From the K\"unneth formula, $H^*(\Z^{d};\Z)$ is the exterior algebra generated by $H^1(\Z^d;\Z)=\Hom(\Z^d;\Z)$. Let $\hat{e}_j$ be the dual basis to the standard basis of $\Z^d$. Then $\widetilde{H}^{\text{even}}(\Z^d;\Z)$ is generated by cup products of elements of the form $\hat{e}_j\smile \hat{e}_i$. Let $\rho^{ij}_n:\Z^{d}\rightarrow U(n)$ be the asymptotic representation defined by sending $e_i$ and $e_j$ to Voiculescu's unitaries and the rest of the generators to the identity. From Theorem~\ref{thm:c11} we see that ${\ch}(\rho_n^{ij})=\hat e_i\smile \hat e_j+n$. Note that by Lemma~\ref{lemma:funct-chern} we can use tensor products and direct sums of $\rho^{ij}_n$ to represent any polynomial with non-negative coefficients in $\hat e_i\smile\hat e_j+n$. It is then sufficient to show that any element of the form 
$$z=\prod_{(i,j)\in S}\hat e_{i}\smile\hat e_j+(3n)^{|S|}$$
is expressible as such a polynomial with non-negative coefficients in $e_i\smile e_j+n$ for any $S\subseteq\{1,\ldots,d\}^2$. Note that since $\hat e_j\smile\hat e_i=-\hat e_i\smile\hat e_j$, it follows that if we may express all such $z$ this way, then we may also express $2(3n)^{|S|}-z$ this way as well. We will show this by induction on $|S|$. For the base case, we have $S=\emptyset$, and what we want to show is trivially true. For the inductive step
\begin{align*}
\prod_{(i,j)\in S\sqcup\{(\ell,m)\}}\hat e_i\smile \hat e_j+(3n)^{|S|+1}=\\
(\hat e_\ell\smile \hat e_m+n)\left(\prod_{(i,j)\in S}\hat e_i\smile \hat e_j+(3n)^{|S|}\right)-&n\left(\prod_{(i,j)\in S}\hat e_i\smile \hat e_j-(3n)^{|S|}\right)-(3n)^{|S|}(\hat{e}_\ell\smile\hat e_m-n).
\end{align*}
By the inductive hypothesis, all the terms on the right-hand side may be expressed in the desired form.

For the last claim, we can use the K\"unneth formula to show that the Chern character
$ \operatorname{ch} : K_0(\mathbb{T}^d) \to H^{\text{even}}(\mathbb{T}^d, \mathbb{Q})$
induces an isomorphism onto $H^{\text{even}}(\mathbb{T}^d, \mathbb{Z}) \subset H^{\text{even}}(\mathbb{T}^d, \mathbb{Q}).$
 Thus, by Corollary~\ref{thm:ex&unique} it follows that all asymptotic representations are stably equivalent to one of the above.
\end{proof}
\begin{remark}
In particular, we can construct explicit asymptotic representations $\rho_n:\Z^{2d}\rightarrow U(m_n)$ such that $\ch_d(\rho_n)$ is the generator of $H^{2d}(\Z^{2d};\Z)$ but $\ch_j(\rho_n)=0$ for $0<j<d$.
\end{remark}

\begin{remark}
In the examples coming from~\ref{formula1}, with $\ch_k\neq 0$,  the defect for any two of the generators  is bounded above by $\frac{2\pi k}{n}$, since this is true for Voiculescu's example and our examples are built out of up to $k$-fold tensor products and direct sums of these. It is possible to calculate the asymptotics of the dimension of the asymptotic representation as well. To approximate an element in the $2d$ cohomology, the dimension is the zeroth Chern character, and can thus be picked to be $(3n)^{d}$. This is not the lowest possible dimension, as we illustrate by Example~\ref{Z6}.
\end{remark}

\begin{notation}
If $\rho:\G\rightarrow U(n)$ then let $\overline{\rho}$ denote the pointwise complex conjugate of $\rho$. As noted in Lemma~\ref{lemma:funct-chern}, if $\rho$ is sufficiently multiplicative for $E_\rho$ to exist, then $E_{\overline{\rho}}=\overline{E_\rho}$, and thus $\ch_k(\overline{\rho})=(-1)^k\ch_k(\rho)$.
\end{notation}

\begin{example}
Let $\pi_1,\pi_2:\Z^4\rightarrow\Z^2$ be projections onto the first two and last two coordinates, respectively. Let $\psi_n: \Z^2\rightarrow U(n)$ be Voiculescu's example of an asymptotic representation. Let $\alpha_n=\psi_n\circ\pi_1$ and $\beta_n=\psi_n\circ\pi_2$. Define $\varphi_n:\Z^4\rightarrow U(3n^2)$ by the formula
$$\varphi_n:=(\alpha_n\tensor\beta_n)\oplus (1_n\tensor\bar\alpha_n)\oplus(1_n\tensor\bar\beta_n).$$
Using Lemma~\ref{lemma:funct-chern}, and the K\"unneth formula, we see that $\ch_2(\varphi_n)$ is the generator of $H^4(\Z^4,\Z)$, while $\ch_1(\varphi_n)=0$.
\end{example}

\begin{example}
Note that $\alpha_n\tensor\beta_n$ is itself a projective representation. By Lemma~\ref{lemma:funct-chern}, one can check that $\ch_2(\alpha_n\tensor\beta_n)$ generates $H^4(\Z^4)$. Thus, if we define 
$$\varphi_n := (\alpha_n \tensor \beta_n )\oplus (\bar\alpha_n \tensor \bar\beta_n),$$
we obtain an asymptotic representation of $\varphi_n:\Z^4\rightarrow U(2n^2)$  so that $\ch_2(\varphi_n)$ is twice the generator of $H^4(\Z^4,\Z)$ but $\ch_1(\varphi_n)=0$.
\end{example}

\begin{example}\label{Z6}
Using analogous notation as above, we define $\pi_1,\pi_2,\pi_3:\Z^6\rightarrow\Z^2$ as projections onto the first two, middle two, and last two coordinates, respectively. Then let $\alpha_n=\psi_n\circ\pi_1$, $\beta_n=\psi_n\circ\pi_2$, and $\gamma_n=\psi_n\circ\pi_3$. We define $\varphi_n:\Z^6\rightarrow U(7n^3)$ by
\begin{align*}
\varphi_n=&(\alpha_n\tensor\beta_n\tensor\gamma_n)\\
&\oplus (1_n\tensor\bar\alpha_n\tensor\beta_n)
\oplus (1_n\tensor\bar\beta_n\tensor\gamma_n)
\oplus (1_n\tensor\bar\gamma_n\tensor\alpha_n)\\
&\oplus (1_{n^2}\otimes\bar\alpha_n)
\oplus (1_{n^2}\otimes\bar\beta_n)
\oplus (1_{n^2}\otimes\bar\gamma_n)
\end{align*}
One verifies by a straightforward computation that 
\[
\ch_3(\varphi_n) = \ch_1(\alpha_n)\ch_1(\beta_n)\ch_1(\gamma_n),
\]
and that the other Chern classes vanish. By the K\"unneth theorem, $\ch_3(\varphi_n)$ generates $H^6(\Z^6,\Z)$.
\end{example}

 \textbf{Concrete almost representations for the group $\Gamma=\HH_3$}.
	Here $\HH_3$ is the discrete Heisenberg group generated by $a,b,c$ with the relations $ba=abc$, $ca=ac$, and $cb=bc$. Each element of the group may be represented uniquely as a triple of integers as follows: $x=(x_1,x_2,x_3)=a^{x_1}b^{x_2}c^{x_3}$. First, we will describe the cohomology ring of $\HH_3$ in terms of explicit generators in terms of the bar resolution. 

\begin{definition}
For $1\le j\le n$ let $e_j$ be the standard basis element of $\C^n$ and extend the symbol by the convention that $e_{j+n}=e_j$. For odd $n$ consider the map $\rho_n:\HH_3\rightarrow U(n)$ determined by the equation

$$\rho_n(x)e_j=\exp\left(\frac{2\pi i}{n}\left(x_3j+\frac12x_2j(j-1)\right)\right)e_{j+x_1}.$$

 This is equivalent to the almost representation due to Eilers, Shulman, and S\o rensen in~\cite{ESS-published}. 

We also define

$$\Tilde{\rho}_n(x)e_j=\exp\left(\frac{2\pi i}{n}\left((x_1x_2-x_3)j+\frac12x_1j(j-1)\right)\right)e_{j+x_2}$$
which is $\rho_n$ composed with the automorphism of $\HH_3$ defined by $(x_1,x_2,x_3)\mapsto(x_2,x_1,x_1x_2-x_3)$.
\end{definition}

 Note that the functions in the exponent are
\begin{align*}
\beta_1(x,y)&=x_3y_1+\frac12x_2y_1(y_1-1)\\
\beta_2(x,y)&=(x_1x_2-x_3)y_2+\frac12x_1y_2(y_2-1),
\end{align*}
which generate $H^2(\HH_3;\Z)$ by Proposition~\ref{H3homology}, so this example is also equivalent to the example formula 
in~\cite[Proposition~4.1]{MR4600056}. Then by Proposition~\ref{chernofprojective} $c_1(\rho_n)=[\beta_1]$ and $c_1(\hat{\rho}_n)=[\beta_2]$. Because $\HH_3$ has a cohomological dimension of 3, all higher Chern classes and characters of $[E_{\rho_n}]$ vanish, or by Theorem~\ref{thm:c11}.

\textbf{Concrete almost representations for the group $\Gamma=\HH_3\times\Z$}.
Let $\pi_1:\G\rightarrow\HH_3$ and $\pi_2:\G\rightarrow\Z$ be the obvious maps. Let $\alpha_1:\HH_3\rightarrow\Z$ and $\alpha_2:\HH_3\rightarrow\Z$ by $(x_1,x_2,x_3)\mapsto x_1$ and $(x_1,x_2,x_3)\mapsto x_2$ respectively. Note that $\psi=(\alpha_2\circ\pi_1,\pi_2)$ is a map from $\Gamma$ to $\Z^2$. Suppose that $\varphi_n:\Z^2\rightarrow U(n)$ is Voiculescu's example of an asymptotic representation of $\Z^2$. 
\begin{proposition}\label{top_form_HtimesZ}
Using the notation above, for odd $n$ define $\eta_n:\G\rightarrow U(3n^2)$ by
$$\eta_n=(\rho_n\circ\pi_1)\tensor(\overline{\varphi_n}\circ\psi)\oplus(\overline{\rho_n}\circ\pi_1)\tensor1_n\oplus(\varphi_n\circ\psi)\tensor1_n.$$
Let $\gamma$ be the generator of $H^3(\HH_3;\Z)$ from Proposition~\ref{H3homology}. Then $c_1([E_{\eta_n}])=0$ and $c_2([E_{\eta_n}])=\pi_1^*([\gamma])\smile\pi_2$. Here $\pi_2\in\Hom(\Gamma,\Z)\cong H^1(\G,\Z)$.
\end{proposition}

\begin{proof}
It is easy to check that $\rho_n$ is a projective representation with cocycle $e^{2\pi i\beta_1}$ and $\varphi_n\circ\psi$ is a projective representation with cocycle $e^{2\pi i (\pi_2\smile \alpha_2\circ\pi_1)}$. Using Theorem~\ref{thm:c11}, Lemma~\ref{lemma:funct-chern}, and Proposition~\ref{H3homology} the result follows from a straightforward computation.
\end{proof}

Note that by the K\"unneth formula, and Proposition~\ref{H3homology}, $\pi_1^*([\gamma])\smile\pi_2$ generates $H^4(\G)$.

\vskip 4pt
\textbf{Concrete almost representations for the group $\G=\HH_3\times\HH_3$}.
 Let $\pi_1$ and $\pi_2$ be the projections from $\G$ to the first and second coordinates, respectively. Define $\psi:\G\rightarrow\Z^2$ as $(\alpha_2\circ\pi_2,\alpha_2\circ\pi_1)$. Define $\varphi_n:\Z^2\rightarrow U(n)$ as in the previous subsection.

\begin{proposition}
Using the notation above, define
\begin{align*}
\eta_n=&(\rho_n\circ\pi_1)\tensor(\varphi_n\circ\psi)\tensor(\rho_n\circ\pi_2)\\
&\oplus(\overline{\rho_n}\circ\pi_1)\tensor(\varphi_n\circ\psi)\tensor1_n\oplus(\overline{\varphi_n}\circ\psi)\tensor(\rho_n\circ\pi_2)\tensor1_n\oplus(\rho_n\circ\pi_1)\tensor(\overline{\rho_n}\circ\pi_2)\tensor1_n\\
&\oplus(\overline{\rho_n}\circ\pi_1)\tensor1_{n^2}\oplus(\overline{\varphi_n}\circ\psi)\tensor1_{n^2}\oplus(\overline{\rho_n}\circ\pi_2)\tensor1_{n^2}.
\end{align*}
Then $c_1([E_{\eta_n}])=0$, $c_2([E_{\eta_n}])=0$, and $c_3([E_{\eta_n}])=2\pi_1^*([\gamma])\smile\pi_2^*([\gamma])$.
\end{proposition}

\begin{proof}
The proof is the same as the proof of Proposition~\ref{top_form_HtimesZ}.    
\end{proof}

Note that $\pi_1^*([\gamma])\smile\pi_2^*([\gamma])$ is the generator of $H^6(\G,\Z)$, by the K\"unneth formula, and Proposition~\ref{H3homology}.

\appendix

\section{Cohomology ring of $\HH_3$}

 The cohomology ring has been computed elsewhere~\cite{HUEBSCHMANN_Nilpotent}, but generators are given in terms of a resolution other than the bar-resolution, which is not well-suited for our present context because our formula for the first Chern class uses the bar resolution. All (co)homology in this appendix is assumed to have coefficients in $\Z$. In particular, our computation of higher invariants relies on the fact that $[\beta_1]\smile[\alpha_2]\ne0$ for the specific cocycles $\beta_1$ and $\alpha_2$ defined below; this does not follow from simply knowing the isomorphism class of the cohomology ring.

\begin{proposition}[\cite{HUEBSCHMANN_Nilpotent},\cite{Packer}]
\begin{align*}
H^1(\HH_3)\cong H_1(\HH_3)&\cong\Z^2\\
H^2(\HH_3)\cong H_2(\HH_3)&\cong\Z^2\\
H^3(\HH_3)\cong H_3(\HH_3)&\cong\Z.
\end{align*}
\end{proposition}

\begin{proposition}\label{H3homology}
The generators of $H^1(\HH_3)$ are given by
\begin{align*}
\alpha_1(x)&=x_1\\
\alpha_2(x)&=x_2.
\end{align*}
The generators of $H^2(\HH_3)$ are given by
\begin{align*}
\beta_1(x,y)&=x_3y_1+\frac12x_2y_1(y_1-1)\\
\beta_2(x,y)&=(x_1x_2-x_3)y_2+\frac12x_1y_2(y_2-1).
\end{align*}
The generator of $H^3(\HH_3)$ is given by
	$$\gamma(x,y,z)=(x_3y_1+\frac12x_2y_1(y_1-1))z_2.$$
The cohomology ring is given by the relations
\begin{align*}
[\alpha_1]\smile[\alpha_2]&=0\\
[\beta_i]\smile[\alpha_j]&=(1-\delta_{ij})\gamma.
\end{align*}
The generators of $H_1(\HH_3)$ are given by 
\begin{align*}
A_1&=[a]\\
A_2&=[b].
\end{align*}
The generators of $H_2(\HH_3)$ are given by
\begin{align*}
B_1&=[c|a]-[a|c]\\
B_2&=[b|c]-[c|b].
\end{align*}
The generator of $H_3(\HH_3,\Z)$ is given by
	
\begin{align*}
		C&=[c^{-1}|bc|ab^{-1}c^{-1}]+[bc|ab^{-1}c^{-1}|a^{-1}bc^{-1}]+[ab^{-1}c^{-1}|a^{-1}bc^{-1}|ac]\\
		&+[a^{-1}bc^{-1}|ac|c^{-1}]-[ac|a^{-1}bc^{-1}|ab^{-1}c^{-1}]-[a^{-1}bc^{-1}|ab^{-1}c^{-1}|bc].
\end{align*}

\end{proposition}

 {Many of these generators were likely already in the literature. The explicit formula for the generator of $H_3(\HH_3)$
may be new.}

The formal proof of Proposition~\ref{H3homology} unfortunately involves a lot of brute force calculations that are not very illuminating. For that reason, we will describe the process by which we computed each generator. Then the computations in the proof of Proposition~\ref{H3homology} show that these do in fact generate the (co)homology of $\HH_3$.

One can compute the 2-cocycles by creating a presentation of a central extension by hand, then computing the multiplication in the extension. These extensions can be built by ``blowing up'' the relations $ac=ca$ and $bc=ca$ respectively. See~\cite[Section 7.2]{arXiv:2204.10354} for a similar computation. 

For any elements $x,y$ in a group that commute with each other it is a useful fact that $[x|y]-[y|x]$ is a 2-cycle. For an abstract motivation of this fact; from the isomorphism in~\cite{iBrown:book-cohomology} we would get $[e|a]+[a|b]-[b|a]-[e|b]$. Then use the fact that $\partial([e|e|b]-[e|e|a])=[e|a]-[e|b]$

The 3-cycle $\gamma$ is just the cup product of $\beta_1$ and $\alpha_2$.

One can compute $C$ as follows. Start with the generator of $H_3(\HH_3)$ in 
the ``small resolution'' given by Huebschmann~\cite{HUEBSCHMANN_Nilpotent}. Following the proof of the independence of resolutions in~\cite[Chapter I.7]{iBrown:book-cohomology}, one can inductively find chain maps from Huebschmann's resolution to the standard resolution. Once the map on the third degree is computed, the image of the generator of $H_3(\HH_3)$ in Huebschmann's resolution is mapped to a generator in terms of the bar resolution. This computation is long, so we will only provide direct computations that demonstrate that $C$ generates the 3-homology.

\begin{proof}
First we compute the 1-cohomology as, $H^1(\HH_3)\cong \mathrm{Hom}(\HH_3,\Z)$ and it is generated by $\alpha_1$ and $\alpha_2$.

The 1-homology is isomorphic to the abelianization, generated by $a$ and $b$. Following the isomorphism from the bar resolution to abelianization, $a$ and $b$ map to $A_1$ and $A_2$ \cite[Chapter II.3]{iBrown:book-cohomology}.

We first check that $\beta_1$ satisfies the cocycle equation;
\begin{align*}
\partial\beta_1(x,y,z)=&\beta_1(x,y)-\beta_1(yz)+\beta_1(xy,z)-\beta_1(y,z)\\
=&x_3y_1+\frac12x_2y_1(y_1-1)\\
&-x_3(y_1+z_1)-\frac12x_2(y_1+z_1)(y_1+z_1-1)\\
&+(x_3+y_3+x_2y_1)z_1+\frac12(x_2+y_2)z_2(z_2-1)\\
&-y_3z_1-\frac12y_2z_1(z_1-1).
\end{align*}
Noting the identity $\frac12(y_1+z_1)(y_1+z_1-1)=\frac12 y_1(y_1-1)+\frac12z_1(z_1-1)+y_1z_1$ one can see that all terms cancel.

Note that $a\mapsto b$, $b\mapsto a$, $c\mapsto c^{-1}$ is an automorphism $\eta$ of $\HH_3$, so that $\eta(x_1,x_2,x_3)=(x_2,x_1,x_1x_2-x_3)$ and $\beta_2=\eta^*(\beta_1)$ so $\beta_2$ is a cocycle as well.

One can see that
$$\partial B_1=[c]-[ac]+[a]-([a]-[ac]+[c])=0$$
and the computation for $B_2$ is identical.

We compute that
\begin{align*}
\langle\beta_1,B_1\rangle&=1\\
\langle\beta_1,B_2\rangle&=0\\
\langle\beta_2,B_1\rangle&=0\\
\langle\beta_2,B_2\rangle&=1.
\end{align*}

We claim that this implies that $\beta_1$ and $\beta_2$ generate $H^2(\HH_3)$ and that $B_1$ and $B_2$ generate $H_2(\HH_3)$. To show this note that the pairing of a cohomology class with both $B_1$ and $B_2$ induces a map $H^2(\HH_3)$ to $\Z^2$. Since this map is surjective, it must also be injective; this relies on previous computations of the (co)homology groups. Thus, it is an isomorphism and so $\beta_1$ and $\beta_2$ generate $H^2(\HH_3)$. The argument that $B_1$ and $B_2$ generate $H_2(\HH_3)$ is analogous.

To show that $[\alpha_1]\smile[\alpha_2]=0$ we will let $\beta_3(x,y)=\alpha_2(x)\alpha_1(y)=x_2y_1$. Then $\beta_3$ is a cocycle representative of $[\alpha_2]\smile[\alpha_1]=-[\alpha_1]\smile[\alpha_2]$. Consider the function $f:\HH_3\rightarrow\Z$ defined by $f(x)=-x_3$. Then 
\begin{align*}
\partial f(x,y)&=f(x)-f(xy)+f(y)\\
&=-x_3+(x_2y_1+x_3+y_3)-y_3\\
&=x_2y_1\\
&=\beta_3(x,y).
\end{align*}
One may also deduce that $[\alpha_1]\smile[\alpha_2]=0$ by using the Gysin sequence~\cite[Theorem 3, Theorem 4]{gysin}.

Note the identity $b^{-1}a^{-1}=a^{-1}b^{-1}c$, in addition to those in the presentation. Using these relations, we compute that all the terms that appear in $\partial C$ are as follows:

\begin{center}
\begin{tabular}{c c c c}
$+[c^{-1}|bc]$& $-[c^{-1}|ac]$ &$+[b|ab^{-1}c^{-1}]$& $-[bc|ab^{-1}c^{-1}]$\\
$+[bc|ab^{-1}c^{-1}]$&$-[bc|c^{-1}]$&$+[ac|a^{-1}bc^{-1}]$&$-[ab^{-1}c^{-1}|a^{-1}bc^{-1}]$\\
$+[ab^{-1}c^{-1}|a^{-1}bc^{-1}]$&$-[ab^{-1}c^{-1}|bc]$&$+[c^{-1}|ac]$&$-[a^{-1}bc^{-1}|ac]$\\
$+[a^{-1}bc^{-1}|ac]$&$-[a^{-1}bc^{-1}|a]$&$+[bc|c^{-1}]$&$-[ac|c^{-1}]$\\
$-[ac|a^{-1}bc^{-1}]$&$+[ac|c^{-1}]$&$-[b|ab^{-1}c^{-1}]$&$+[a^{-1}bc^{-1}|ab^{-1}c^{-1}]$\\
$-[a^{-1}bc^{-1}|ab^{-1}c^{-1}]$&$+[a^{-1}bc^{-1}|a]$&$-[c^{-1}|bc]$&$+[ab^{-1}c^{-1}|bc]$
\end{tabular} 
\end{center}

One can check that each of these terms cancel out, so $\partial C=0$.

By construction we can see that $[\gamma]=[\beta_1]\smile[\alpha_2]$, and from this deduce that $\gamma$ obeys the cocycle equation. Going term-by-term, we compute that
$$\langle\gamma,C\rangle=0+1+0+0-1-(-1)=1$$
Because $H_3(\HH_3)\cong H^3(\HH_3)\cong\Z$ and
$\langle\gamma,C\rangle=1$
we can see that (co)homology classes of $C$ and $\gamma$ generate the $H_3(\HH_3)$ and $H^3(\HH_3)$ respectively.

To show the other facts about the cup product we define
\begin{align*}
\gamma_{1,1}(x,y,z)&=(x_3y_1+\frac12x_2y_1(y_1-1))z_1\\
\gamma_{2,1}(x,y,z)&=((x_1x_2-x_3)y_2+\frac12x_1y_2(y_2-1))z_1\\
\gamma_{2,2}(x,y,z)&=((x_1x_2-x_3)y_2+\frac12x_1y_2(y_2-1))z_2
\end{align*}
so that $\gamma_{i,j}$ is a cocycle representative of $[\beta_i]\smile[\alpha_j]$. Then
\begin{align*}
\langle\gamma_{1,1},C\rangle&=0+(-1)+0+0-(-1)-0=0\\
\langle\gamma_{2,1},C\rangle&=1+(-1)+0+0-(-1)-0=1\\
\langle\gamma_{2,2},C\rangle&=(-1)+1+0+0-1-(-1)=0.
\end{align*}

Since the map given by the universal coefficient theorem from $H^3(\HH_3)$ to $\Hom(H_3(\HH_3),\Z)$ is an isomorphism in this case we conclude that the cohomology class of 3-cocycles on $\HH_3$ is determined by their pairing with $C$. In particular, this implies the cup product structure claimed in the introduction is correct.

\end{proof}

\textbf{Acknowledgments} The second author thanks Rufus Willett for a stimulating discussion that led to some steps of the proof of Proposition~\ref{prop:projB}.

\end{document}